%% file: witt_steenrod.tex
\pdfoutput=1

\documentclass{amsart}
\usepackage{verbatim}
\usepackage{bbm}
\usepackage{amscd}
\usepackage{stmaryrd}
\usepackage[margin=1in]{geometry}
\usepackage{microtype}
\usepackage{amsmath,amsfonts,amssymb,amsthm}
\usepackage{mathtools}
\usepackage{eucal}
\usepackage{tikz-cd}
\usepackage{tikz}
\usepackage{soul}
\usepackage{longtable}
\usepackage{hyperref}
\usepackage{enumerate}

\definecolor{darkred}{RGB}{128,0,0}
\definecolor{darkgreen}{RGB}{0,128,0}
\definecolor{darkblue}{RGB}{0,0,128}

\hypersetup{linktocpage,
	pdfborder = {0 0 0},
	colorlinks,
	citecolor=darkgreen,
	filecolor=darkred,
	linkcolor=darkblue,
	urlcolor=cyan!50!black!90}

\newtheoremstyle{exercise}							
{15pt}											
{15pt}											
{\rm}												
{}												
{\bfseries}											
{}												
{10pt}											
{}												

\numberwithin{equation}{section}
\newtheorem{proposition}[equation]{Proposition}
\newtheorem{theorem}[equation]{Theorem}
\newtheorem{corollary}[equation]{Corollary}
\newtheorem{lemma}[equation]{Lemma}
\theoremstyle{exercise}
\newtheorem{remark}[equation]{Remark}
\newtheorem{example}[equation]{Example}

\newtheorem{notation}[equation]{Notation}

\newcommand{\C}{\mathbb{C}} 						
\newcommand{\Q}{\mathbb{Q}} 						
\newcommand{\Z}{\mathbb{Z}} 						
\newcommand{\N}{\mathbb{N}} 						
\newcommand{\Ss}{\mathbb{S}} 						
\newcommand{\Gm}{\mathbb{G}_m}  					

\newcommand{\spec}{\mathrm{Spec}} 					
\DeclareMathOperator{\cofib}{cofib} 					
\DeclareMathOperator{\ra}{\rightarrow} 					
\DeclareMathOperator{\lra}{\longrightarrow}				
\newcommand{\xra}[1]{\xrightarrow{#1}} 					
\newcommand{\xla}[1]{\xleftarrow{#1}} 					
\DeclareMathOperator{\hra}{\hookrightarrow}				
\newcommand{\thra}{\twoheadrightarrow}	
\newcommand{\xhra}[1]{\overset{#1}{\hookrightarrow}}		
\DeclareMathOperator{\im}{im}							
\DeclareMathOperator{\vcd}{vcd}					
\DeclareMathOperator{\cd}{cd}					
\DeclareMathOperator{\Hom}{\mathrm{Hom}}				
\DeclareMathOperator{\SH}{\mathcal{SH}} 				
\DeclareMathOperator{\Smk}{\mathsf{Sm_k}}				
\newcommand{\unit}{\mathbf{1}} 						
\DeclareMathOperator{\kw}{kw} 							

\DeclareMathSymbol{A}{\mathalpha}{operators}{`A}
\DeclareMathSymbol{B}{\mathalpha}{operators}{`B}
\DeclareMathSymbol{C}{\mathalpha}{operators}{`C}
\DeclareMathSymbol{D}{\mathalpha}{operators}{`D}
\DeclareMathSymbol{E}{\mathalpha}{operators}{`E}
\DeclareMathSymbol{F}{\mathalpha}{operators}{`F}
\DeclareMathSymbol{G}{\mathalpha}{operators}{`G}
\DeclareMathSymbol{H}{\mathalpha}{operators}{`H}
\DeclareMathSymbol{I}{\mathalpha}{operators}{`I}
\DeclareMathSymbol{J}{\mathalpha}{operators}{`J}
\DeclareMathSymbol{K}{\mathalpha}{operators}{`K}
\DeclareMathSymbol{L}{\mathalpha}{operators}{`L}
\DeclareMathSymbol{M}{\mathalpha}{operators}{`M}
\DeclareMathSymbol{N}{\mathalpha}{operators}{`N}
\DeclareMathSymbol{O}{\mathalpha}{operators}{`O}
\DeclareMathSymbol{P}{\mathalpha}{operators}{`P}
\DeclareMathSymbol{Q}{\mathalpha}{operators}{`Q}
\DeclareMathSymbol{R}{\mathalpha}{operators}{`R}
\DeclareMathSymbol{S}{\mathalpha}{operators}{`S}
\DeclareMathSymbol{T}{\mathalpha}{operators}{`T}
\DeclareMathSymbol{U}{\mathalpha}{operators}{`U}
\DeclareMathSymbol{V}{\mathalpha}{operators}{`V}
\DeclareMathSymbol{W}{\mathalpha}{operators}{`W}
\DeclareMathSymbol{X}{\mathalpha}{operators}{`X}
\DeclareMathSymbol{Y}{\mathalpha}{operators}{`Y}
\DeclareMathSymbol{Z}{\mathalpha}{operators}{`Z}
\DeclareMathSymbol{q}{\mathalpha}{operators}{`q}

\begin{document}
	\title{The Dual Motivic Witt Cohomology Steenrod Algebra}
	\date{\today}
	\author{Viktor Burghardt}	
	\begin{abstract}
		In this paper we begin the study of the (dual) Steenrod algebra of the motivic Witt cohomology spectrum $H_W\Z$ by determining the algebra structure of ${H_W\Z}_{**}H_W\Z$ over fields $k$ of characteristic not $2$ which are extensions of fields $F$ with $K^M_2(F)/2=0$. For example, this includes all fields of odd characteristic, as well as fields that are extensions of quadratically closed fields of characteristic $0$. After inverting $\eta$, this computes the $HW:=H_W\Z[\eta^{-1}]$-algebra ${HW}_{**}HW$. In particular, for the given base fields, this implies the $HW$-module structure of $HW\wedge HW$ which has been computed by Bachmann and Hopkins in \cite{bachmann-etaperiodic-fields}\cite{bachmann-etaperiodic-dedekind}.
	\end{abstract}
	\maketitle
	\tableofcontents
	\input{intro}
	\input{preliminaries}
	\input{motivic_steenrod}
	\input{h_w-motives}
	\nocite{*}
	\bibliographystyle{halpha}
	\bibliography{bib}
\end{document}

%% file: intro.tex
\section{Introduction}
	In classical homotopy theory, the determination of the structure of the Steenrod algebra $\mathcal{A}_p$ was key to the development of the field. It lead, among other results, to Adams' resolution of the Hopf invariant one problem \cite{hopfinvariantone}, Thom's classification  of smooth manifolds up to cobordism \cite{Thom1954}, accessibility of the Adams and Adams-Novikov spectral sequences \cite{Adams1957/58} \cite{NovikovSS} and the many computations that come with it. 
		
	In motivic homotopy theory, Voevodsky computed the Steenrod algebra, and its dual, of motivic cohomology mod $p$ \cite{Voevodsky_reduced_power_operations_in_motivic_cohomology}. For $p=2$ the dual motivic Steenrod algebra is given by
	$$H\Z/2_{**}H\Z/2\cong H\Z/2_{**}[\tau_0,\tau_1,\ldots,\xi_1,\xi_2,\ldots]/(\tau_i^2-\rho\tau_{i+1}-(\tau+\rho\tau_0)\xi_{i+1}).$$ 
	It was central to the establishment of the Milnor, and more generally Bloch-Kato, conjectures, now also known as the norm-residue isomorphism. The resolution of the Milnor conjecture \cite{Voevodsky-mot_coh_with_2-coef} and the Beilinson-Lichtenbaum conjecture \cite[Section 6]{suslinvoevodskyblochkato} implies, see for example \cite[7.2]{kylling_hermktheory}, that the coefficient ring of $H\Z/2$ is given by a polynomial ring over Milnor K-theory mod $2$ in one variable $\tau$ with degree $|\tau|=(0,-1)$:  
	$$H\Z/2_{**}\cong \underline{K}^M_{**}/2[\tau].$$ 
	Let $\SH(k)$ be the motivic stable homotopy category of Morel and Voevodsky over a perfect field $k$ as described in \cite{morel_intro_to_A1}. Given a spectrum $E\in \SH(k)$, we denote by $\pi_i(E)_*:=\pi_{i-*,-*}E$ the \emph{i-line} of the stable homotopy groups of $E$. The $0$-line of the stable homotopy groups of the motivic sphere spectrum is given by Milnor-Witt K-theory \cite{morela1algtopoverfield} \cite{morel_motivic_pi0}  
	$$\pi_0(\Ss)_*\cong K_*^{MW}(k).$$ 
	As a consequence of the Milnor conjecture, it decomposes via the pullback square \cite[Thm 5.3]{morel_sur_les_puissances_de_I}\cite[Chapter 5]{morela1algtopoverfield}
	\begin{equation}\label{ktheorypullback}
		\begin{tikzcd}
			K_*^{MW}(k) \ar[r] \ar[d]& K_*^M(k) \ar[d]\\
			K_*^W(k) \ar[r] & k_*^M(k),
		\end{tikzcd}
	\end{equation}
	where $K_*^M$ denotes the graded ring of Milnor K-theory, $K_*^W$ denotes the graded ring of Witt K-theory and $k_*^M$ denotes the graded ring of Milnor K-theory modulo $2$. In degree $(0,0)$ this is the pullback square
	\[\begin{tikzcd}
		GW(k) \ar[r] \ar[d]& \Z \ar[d]\\
		W(k) \ar[r] & \Z/2.
	\end{tikzcd}\]
	It relates the Grothendieck-Witt ring $GW(k)$ to the Witt ring of symmetric bilinear forms $W(k)$. These are theories very arithmetic in nature. 
	Let us consider this at the level of motivic Eilenberg-Maclane spectra. The square (\ref{ktheorypullback}) can be realized as homotopy groups of the pullback square of the corresponding homotopy modules in the heart of $\SH(k)$ with respect to Morel's homotopy $t$-structure \cite{morel_intro_to_A1}. Taking effective covers one arrives at a homotopy pullback square 
	\begin{equation}\label{cohomologypullback}
		\begin{tikzcd}
			H\tilde{\Z} \ar[r] \ar[d] & H\Z \ar[d]\\
			H_W\Z \ar[r] & H\Z/2.
		\end{tikzcd}
	\end{equation} 
	Motivic cohomology, represented by the motivic spectrum $H\Z\in\SH(k)$, is in various ways similar to singular cohomology. The spectrum $H\Z/2$ is the mod $2$ version of $H\Z$. The other two spectra, motivic Milnor-Witt cohomology $H\tilde{\Z}$ and motivic Witt cohomology $H_W\Z$, are less classical and arise from new phenomena in $\SH(k)$. People have intensively studied $\eta$-inverted variants of $H_W\Z$, $HW:=H_W\Z[\eta^{-1}]$, the structure of $HW\wedge HW$, as well well as structures of $\eta$-inverted relatives like hermitian K-theory \cite{bachmann_generalized-slices}\cite{roendigs_cellularity_of_hermktheory}\cite{ananyevskiy_stableopinderwitt}\cite{ananyevskiy_very_effective_herktheory}\cite{kylling_hermktheory} \cite{bachmann-etaperiodic-fields}\cite{bachmann-etaperiodic-dedekind}\cite{fasel2020stable}.
		
	Our goal is to study the spectra on the left hand side of the square (\ref{cohomologypullback}), specifically the algebras of self-cohomology (Steenrod algebra) and their self-homology (dual Steenrod algebra). We hope this will shed some light on the Witt part of $\SH(k)$. In this paper we start by determining the algebra structure of ${H_W\Z}_{**}H_W\Z$ for certain base fields. We plan this to be the first in a series of studies to determine the entire structure.
	
	The coefficient ring ${H_W\Z}_{**}$ is known. Combining our knowledge of the coefficient ring $H\Z/2_{**}$ with the resolution of Morel's structure conjecture \cite[Thm. 17]{bachmann_generalized-slices}, we get
	$${H_W\Z}_{**}\cong \underline{K}_{**}^W[\tau]/(\tau \eta).$$
	The element $\eta$ is already present in $\pi_{1,1}\Ss$ and originates from the unstable Hopf map $\mathbb{A}^2\setminus {0}\ra \mathbb{P}^1$.
		
	We now describe the main theorem of this paper. Let $I=(I_2,I_3,\ldots)$ be a sequence with $I_i\in \{0,1\}$ and almost all $I_i=0$ and denote the zero sequence by $\emptyset$. Then we define
	\begin{gather*}
		c(I):=\sum_{i\in I}\overline{\xi_{i-1}}^2\overline{\xi}(I\setminus i), \hspace{0.7cm} c_1(I):=\tau_0\overline{\xi}(I)+\tau_1c(I),\\
		t(I):=\tau \overline{\xi}(I),\hspace{1cm} t_1(I):=\tau \tau_1\overline{\xi}(I),
	\end{gather*}
	where a line above an element stands for a conjugated element via the conjugation $\iota:H\Z/2\wedge H\Z/2\xra{\simeq} H\Z/2\wedge H\Z/2$. There exist canonical lifts of these elements under the canonical map ${H_W\Z}_{**}H_W\Z\ra H\Z/2_{**}H\Z/2$. We will abuse notation and denote the lifted elements also as defined above. Furthermore, there exist lifts $t_j\in {H_W\Z}_{**}H_W\Z$ of $\tau_j\in H\Z/2_{**}H\Z/2$ for $j\geq 2$ and $s\in {H_W\Z}_{**}H_W\Z$ of $\tau_0^3\tau_1\in H\Z/2_{**}H\Z/2$. With this notation, the main theorem takes the following form.
	\begin{theorem}[see Thm. \ref{mainthm}]\label{intro_thm}
		Let the base field $k$ be an extension of a field $F$ satisfying $k^M_2(F)=0$. Then we have an isomorphism of right ${H_W\Z}_{**}$-algebras
		$${H_W\Z}_{**} {H_W\Z} \cong {H_W\Z}_{**}[s,t_2,t_3,\ldots,c(I),t(I),c_1(I),t_1(I)]/(\ldots),$$
		where $I$ is a binary index set starting at index $2$.  The ideal of $\eta$-torsion is generated by $c(I),c_1(I),t(I),t_1(I)$. The ideal $(\ldots)$ is generated by the relations in Lemma \ref{c_relations}, Lemma \ref{t_relations} and
		\begin{gather*}
		c(\emptyset)=1,\hspace{1cm} \tau_0^4=c(e_{2})\tau^2,\hspace{1cm}t_j^2=\rho t_{j+1}+\xi_{j+1}\tau,\hspace{1cm}s^2=\tau_0^2(\rho t_2+\xi_2\tau)c(e_2)\tau^2,\\
		sc_1(I)=\tau_0^3(\rho t_2+\xi_2\tau)c(I)+\tau t_1(I)c(e_2),\hspace{1cm}t(I)s=t_1(I)\tau_0^3,\hspace{1cm} t_1(I)s=\tau_0^3t(I)(\rho t_2+\xi_2\tau).
		\end{gather*}
		The degrees here are
		\begin{gather*}
		|t_1|=(3,0), \hspace{1cm} |s|=(5,0),\hspace{1cm}|t_i|=(2^{i+1}-1,2^i-1),\\
		|c(I)|=\sum_{i\in I} (2^{j+2}-4,2^{j+1}-2), \hspace{1cm} |c_1(I)|=\sum_{i\in I} (2^{j+2}-1,2^{j+1}-1),\\
		|t(I)|=(0,-1)+\sum_{i\in I} (2^{i+1}-2,2^{i}-1), \hspace{1cm} |t_1(I)|=(3,0)+\sum_{i\in I} (2^{i+1}-2,2^{i}-1).
		\end{gather*}
	\end{theorem}
	
	The theorem shows that ${H_W\Z}_{**}H_W\Z$ has $\eta$-torsion. Thus, $H_W\Z_{**} H_W\Z$ is not free as an $H_W\Z$-module. It will therefore not be the dual of $H_W\Z^{**}H_W\Z$ in the category of $H_W\Z^{**}$-modules. The existence of $\eta$-torsion makes the computation a lot more difficult. One could compare the phenomena in this case to similar phenomena arising in the classical computation of $bo_{**}bo$ \cite{mahowaldboresolutions}. Along the way we show that there is no higher $\eta$-torsion. A fact we believe will further become useful in the future analysis of $H_W\Z^{**}H_W\Z$.
	
	After inverting $\eta$, Theorem \ref{intro_thm} computes the $HW:=H_W\Z[\eta^{-1}]$-algebra $HW_{**}HW$. 
	\begin{corollary}[see Cor. \ref{eta_inverted_algebra}]
		Let the base field $k$ be an extension of a field $F$ satisfying $k^M_2(F)=0$. Then we have an isomorphism of $W(k)[\eta^{-1}]$-algebras
		$${HW}_{**}HW\cong W(k)[\eta^{\pm 1}][x_2,y,x_3,x_4,\ldots]/(y^2,x_j^2-2x_{j+1}),$$
		where $|x_j|=(2^j,0)$ and $|y|=(5,0)$.
	\end{corollary} 
	In particular, over base fields with the given assumptions, this shows the result \cite[Cor. 8.20]{bachmann-etaperiodic-fields} of Bachmann and Hopkins about the $HW$-module structure of $HW\wedge HW$. 
		
	\subsection{Outline}
		To establish our result we will make heavy use of the homotopy pullback square of Proposition \ref{fundamental-pullback-square} in $\SH(k)$
		\[\begin{tikzcd}
			H_W\Z \ar[r,"r"] \ar[d,"\pi_W"] & H\Z/2 \ar[d,"\pi"]\\
			\underline{K}^W \ar[r,"\underline{r}"] & \underline{k}^M.
		\end{tikzcd}\]
		We smash with $H_W\Z$ and $H\Z/2$ to get two more pullback squares
		\[\begin{tikzcd}
			H_W\Z\wedge H_W\Z \ar[r] \ar[d] & H\Z/2\wedge H_W\Z \ar[d]\\
			\underline{K}^W\wedge H_W\Z \ar[r] & \underline{k}^M\wedge H_W\Z
		\end{tikzcd}
		\hspace{2cm}
		\begin{tikzcd}
			H\Z/2\wedge H_W\Z \ar[r] \ar[d] & H\Z/2\wedge H\Z/2 \ar[d]\\
			H\Z/2\wedge \underline{K}^W \ar[r] & H\Z/2\wedge \underline{k}^M.
		\end{tikzcd}\]
		We want to get to the object at the top left corner of the left square. To do so we compute homotopy groups of all the others first. Clockwise we start at $H\Z/2\wedge H_W\Z$ via the right square. There we know the top right and can easily determine the bottom right by setting $\tau\in H\Z/2_{**}$ to zero in $H\Z/2_{**}H\Z/2$. The bottom left corner of the right square can be computed using the well-known cofiber sequence \cite{morel_sur_les_puissances_de_I}\cite[Remark 3.12]{morel_intro_to_A1}
		$$\Sigma^{1,1}\underline{K}^W\xra{\eta} \underline{K}^W\ra \underline{k}^M.$$ 
		After smashing with $H\Z/2$, this splits as spectra in $\SH(k)$ and the map that continues to the right will be determined by the reduced power operation $Sq^2\in H\Z/2^{**}H\Z/2$, the dual element of $\xi_1\in H\Z/2_{**}H\Z/2$. Thanks to the known structure of the motivic Steenrod algebra and its dual we can carry out this computation. A pullback square gives us $H\Z/2_{**}H_W\Z$. Again, modding out $\tau$ we can establish the bottom right corner of the left square. We attack the bottom left corner of the left square similarly via the above mentioned cofiber sequence. This time we smash with $H_W\Z$, but do not get a splitting. However, it can be shown that there is no higher $\eta$-torsion in the picture. So the long exact sequence of homotopy groups provides enough information to conclude. To deal with $\eta$-completion, we use a result by Bachmann and Hopkins \cite[Thm. 5.1]{bachmann-etaperiodic-fields} which relates $2$ and $\eta$-completion. We then conclude by pulling back the information we gathered as shown in the left square.
		
	\subsection{Acknowledgments}
		We would like to thank Elden Elmanto for bringing up the idea to study the Steenrod algebra of $H_W\Z$ and for answering various questions regarding motivic homotopy theory during our time at Northwestern. We thank Tom Bachmann and Jonas Kylling for sharing their ideas and notes on the subject. Section \ref{sec_H_H_W} was already known to Jonas and outlined in his notes. Furthermore, the author is grateful to Paul Arne {\O}stv{\ae}r for inviting them to spend some time at the University of Oslo under the MHE project during which part of the narrative of this paper was formed. We also thank Ivan Panin and Andrei Druzhinin for endless discussions during the time in Oslo. Thanks further goes to Xiao Wang for his support. Last but not least, the author thanks Paul Goerss for the infinite amount of helpful discussions on the matter and his guidance as an advisor throughout graduate school.
			
	\subsection{Conventions}
		We will generally assume that $k$ is a field of characteristic not $2$, unless stated otherwise. $\SH(k)$ will mean the motivic stable homotopy category of $\mathbb{P}^1$-spectra as described in \cite{morel_intro_to_A1}. Let $E, F, X\in \SH(k)$ be spectra. We write $E_{**}F=\pi_{**}(E\wedge F)$ and $F_{**}E=\pi(F\wedge E)$. Given a map $f:E\ra F$, we denote by $f_*=f_*^R:X_{**}E\ra X_{**}F$ the map $\pi_{**}(X\wedge f)$ and by $f_*^L:E_{**}X\ra F_{**}X$ the map $\pi_{**}(f\wedge X)$. The homotopy groups $\pi_{i-j,-j}E$ will also be denoted by $\pi_i(E)_j$ to emphasize that they form the $i$-line of stable homotopy groups of $E$, i.e. the line with slope $1$ that contains the point $(i,0)$.
		\begin{longtable}{lll}
			$K_*^{MW}$ & graded ring of Milnor-Witt K-theory\\
			$K_*^W$ & graded ring of Witt K-theory\\
			$K_*^M$ & graded ring of Milnor K-theory\\
			$k_*^M$ & graded ring of Milnor K-theory mod 2\\
			$\underline{K}^{MW}$ & homotopy module of Milnor-Witt K-theory\\
			$\underline{K}^W$ & homotopy module of Witt K-theory\\
			$\underline{K}^M$ & homotopy module of Milnor K-theory\\
			$\underline{k}^M$ & homotopy module of Milnor K-theory mod 2\\
			$H\tilde{\Z}$ & Milnor-Witt motivic cohomology ring spectrum\\
			$H\Z$ & motivic cohomology ring spectrum\\
			$H_W\Z$ & Witt motivic cohomology ring spectrum\\
			$H\Z/2$ & motivic cohomology mod 2\\
			$HW=H_W\Z[\eta^{-1}]$ & $\eta$-inverted Witt motivic cohomology ring spectrum\\
			$E^\wedge_x$ & $x$-completion of a module or a spectrum $E$\\
			$\vcd2(k)$ & étale $2$-cohomological dimension of the field $k[\sqrt{-1}]$\\
			$\pi_iE_j$ & bigraded stable homotopy groups of a spectrum $E$ with $\pi_iE_j=\pi_{i-j,-j}E$\\
			$\N$ & natural numbers without zero\\
			$\N_0$ & natural numbers and $0$.
		\end{longtable}

%% file: preliminaries.tex
\section{Preliminaries}
	This section recalls known results and defines some objects we will work with. It in particular recalls the definitions of the objects mentioned in the introduction.
	
	\subsection{The Homotopy t-Structure}
	Let $k$ be a field of characteristic not $2$. Morel has defined a $t$-structure on $\SH(k)$ \cite{morel_intro_to_A1}, the \emph{homotopy t-structure}. The corresponding subcategories of $n$-connective spectra $\SH(n)_{\geq n}$ and $n$-coconnective spectra $\SH(n)_{\leq n}$ can be described by \cite[Thm. 2.3]{hopkins_morel}
	\begin{gather*}
	E\in \SH(k)_{\geq n} \hspace{0.5cm}\longleftrightarrow\hspace{0.5cm} \underline{\pi}_{p,q}(E)=0 \text{ for } p-q < n,\\
	E\in \SH(k)_{\leq n} \hspace{0.5cm}\longleftrightarrow\hspace{0.5cm} \underline{\pi}_{p,q}(E)=0 \text{ for } p-q > n.
	\end{gather*}
	The objects $\underline{\pi}_{p,q}(E)$ on the right hand side are the \emph{bigraded homotopy sheaves} of $E$. They are given as the Nisnevich sheafification of the presheaves of smooth schemes over $k$ defined by $X\mapsto [\Sigma^{p,q}\Sigma^\infty X_+,E]$. When the base $k$ is a perfect field, Morel showed that homotopy sheaves detect equivalences on finitely generated field extensions $k\hra L$, see \cite[Thm. 2.7]{hopkins_morel}. We will denote for a motivic spectrum $E\in \SH(k)$ the groups
	$$\pi_{p,q}(E)=[\Ss^{p,q},E],$$
	where $\Ss=\Sigma^\infty\spec(k)_+$ is the sphere spectrum. The heart of the homotopy t-structure is given by the abelian category $\SH(k)^\heartsuit$ of \emph{homotopy modules}. By the above description of the t-structure, the only non-vanishing homotopy sheaves of $F\in \SH(k)^\heartsuit$ are on the $0$-line, i.e. $\underline{\pi}_0(F)_*=\underline{\pi}_{-*,-*}(F)$. Since the sphere spectrum is connective its coconnective cover $\Ss_{\leq 0}$ lies in $\SH(k)^\heartsuit$. By Morel's computation its bigraded homotopy groups are given by Milnor-Witt K-theory \cite{morela1algtopoverfield} \cite{morel_motivic_pi0}  
	$$\pi_0(\Ss)_*\cong K_*^{MW}(k).$$ 
	It thus deserves the name $\underline{K}^{MW}:=\Ss_{\leq 0}$, the \emph{Milnor-Witt homotopy module}. Generally, underlines in the name of a spectrum will indicate that the spectrum lies in the heart $\SH(k)^\heartsuit$. One further defines, by taking cones in the abelian category $\SH(k)^\heartsuit$, the \emph{Milnor}, \emph{Witt} and \emph{Milnor mod 2 homotopy modules} \cite[Example 3.33]{morel_intro_to_A1}
	\begin{align*}
	\underline{K}^M&:=\underline{K}^{MW}/\eta,\\
	\underline{K}^W&:=\underline{K}^{MW}/h,\\
	\underline{k}^M&:=\underline{K}^{W}/\eta.
	\end{align*}
	In the above $\eta:\Ss^{1,1}\ra \Ss$ is a stable version of the unstable Hopf map $\mathbb{A}^2\setminus {0}\ra \mathbb{P}^1$, and $h:\Ss\ra \Ss$ is defined as $h := 2 + \eta\rho$, with $\rho:\Ss^{-1,-1}\ra \Ss$ being a stable version of the map $\spec(k) \ra \Gm$ corresponding to $-1\in k^\times$.
	
	It follows from \cite{morel_sur_les_puissances_de_I}\cite[Remark 3.12]{morel_intro_to_A1} that we have a cofiber sequence of motivic spectra 
	\begin{gather}\label{cofiber_kW_k}
	0\ra \Sigma^{1,1}\underline{K}^W\xra{\eta} \underline{K}^W \xra{\underline{r}}\underline{k}^M\ra 0.
	\end{gather}
	The homotopy groups of these homotopy modules are given by 
	$$\pi_0(\underline{K}^{MW})_*\cong K^{MW}_*(k),\hspace{0.5cm}\pi_0(\underline{K}^M)_*\cong K^M_*(k) ,\hspace{0.5cm} \pi_0(\underline{K}^W)_*\cong K^W_*(k), \hspace{0.5cm} \pi_0(\underline{k}^M)_*\cong k^M_*(k)$$
	on the diagonal and by $0$ otherwise. They will be denoted by $\underline{K}^{MW}_{**}$, etc., as is usual for spectra.
	
	\subsection{Effective and Very Effective Spectra}\label{sec_eff_and_veff}
	Voevodsky defined \cite{voevodsky_openproblems} the subcategory of \emph{effective spectra} $\SH(k)^\text{eff}\subset \SH(k)$ which can be given as the full subcategory generated by homotopy colimits and extensions by suspension spectra $\Sigma^{p,q}X_+$ of smooth schemes $X\in \Smk$ and $p\in \Z,q\geq 0$. The inclusion $i_t:\Sigma^{0,t}\SH(k)^\text{eff}\wedge\hra\SH(k)$ has a right adjoint $r_t:\SH(k)\ra  \Sigma^{0,t}\SH(k)^\text{eff}$. The composition $f_t:=i_t\circ r_t$ is called the $t$-th \emph{effective cover} functor. It preserves homotopy colimits and there are natural transformations $f_{t+1}\ra f_t$. The \emph{$t$-th slice} of a spectrum $E\in \SH(k)$ is defined by the cofiber sequence
	$$f_{t+1}(E)\ra f_t(E)\ra s_t(E).$$ 
	With the help of the effective cover we furthermore define
	\begin{align*}
	H\tilde{\Z}&:=f_0(\underline{K}^{MW}),\\
	H\Z&:=f_0(\underline{K}^M),\\
	H_W\Z&:=f_0(\underline{K}^W),\\
	H\Z/2&:=\cofib(2:H\Z\ra H\Z).
	\end{align*}
	It was shown in \cite{bachmann2018effectivity} that $H\Z$ defined as above represents motivic cohomology. The coefficient ring $H\Z/2_{**}$ is known due to the resolution of the Milnor conjecture \cite{Voevodsky-mot_coh_with_2-coef} and the Beilinson-Lichtenbaum conjecture \cite[Section 6]{suslinvoevodskyblochkato}, see for example \cite[7.2]{kylling_hermktheory}. It is given as a polynomial ring over Milnor K-theory mod $2$ \cite{milnoralgktheoryandquadforms} in one variable $\tau\in H\Z/2_{0,-1}$
	$$H\Z/2_{**}\cong \underline{k}^M_{**}[\tau].$$
	More precisely, $H\Z/2_{p,q}\cong k_p^M(k)\tau^{q-p}$ for $p\geq q$ and $0$ otherwise. This implies that there is a cofiber sequence
	\begin{equation}\label{tau_cofiber_seq}
		\Sigma^{0,-1}H\Z/2 \xra{\tau} H\Z/2\ra \underline{k}^M.
	\end{equation} 
	Together with the resolution of Morel's structure conjecture \cite[Thm. 17]{bachmann_generalized-slices} this computes ${H_W\Z}_{**}$ as 
	$${H_W\Z}_{**}\cong \underline{K}_{**}^W[\tau]/(\tau \eta),$$
	and we have a cofiber sequence
	$$\Sigma^{0,-1}H\Z/2 \xra{\tilde{\tau}} H_W\Z\ra \underline{K}^W.$$
	
	Bachmann defined \cite{bachmann_generalized-slices} a t-stucture on $\SH(k)^\text{eff}$ via
	\begin{gather*}
	E\in \SH(k)^\text{eff}_{\geq n} \hspace{0.5cm}\longleftrightarrow\hspace{0.5cm} \underline{\pi}_{p}(E)_0=0 \text{ for } p < n,\\
	E\in \SH(k)^\text{eff}_{\leq n} \hspace{0.5cm}\longleftrightarrow\hspace{0.5cm} \underline{\pi}_{p}(E)_0=0 \text{ for } p > n.
	\end{gather*}
	Denote the non-negative part $\SH(k)^\text{eff}_{\geq 0}$ with respect of this t-structure by $\SH(k)^\text{veff}$; it was first defined in \cite{spitzweck2012}. This is a full subcategory of $\SH(k)$ stable under homotopy colimits and extensions and its objects are called \emph{very effective} spectra. We say a spectrum $E$ is \emph{very n-effective}, if $\Sigma^{-n,0}E$ is very effective.
	\begin{proposition}\label{ourpsectraareveff}
		The spectra $H\tilde{\Z},H\Z,H_W\Z,H\Z/2$ and $\Sigma^{1,1}\underline{K}^W$ as well as any combination of their smash products are very effective.
	\end{proposition}
	\begin{proof}
		Any effective spectrum that is connective with respect to the homotopy t-structure is clearly very effective. Furthermore, $\SH(k)^\text{eff}$ is stable under smash products \cite[Lemma 5.6]{spitzweck2012}. It remains to show that $\Sigma^{1,1}\underline{K}^W$ is very effective. Taking effective covers of (\ref{cofiber_kW_k}) we get a cofiber sequence
		$$f_0(\Sigma^{1,1}\underline{K}^W)\xra{\tilde{\eta}}H_W\Z\xra{r}H\Z/2.$$
		Since $H_W\Z$ and $H\Z/2$ are very effective, so is $f_0(\Sigma^{1,1}\underline{K}^W)$. The resolution of Morel's structure conjecture \cite[Thm. 17]{bachmann_generalized-slices} implies that $f_0(\Sigma^{1,1}\underline{K}^W)\simeq \Sigma^{1,1}\underline{K}^W$. The claim follows.
	\end{proof}

	\begin{remark}\label{cellularity}
		When the characteristic of the base field is not $2$, It has been shown \cite{hopkins_morel}\cite{bachmann-etaperiodic-fields}\cite{bachmann-etaperiodic-dedekind} that the spectra $\underline{K}^{MW}, \underline{K}^W,\underline{K}^M,\underline{k}^M$ and their effective covers $H\tilde{\Z},H\Z,H_W\Z,H\Z/2$ are cellular. Therefore, for all our purposes it suffices to check equivalences on bigraded homotopy groups $\pi_{p,q}$ \cite[Prop. 7.1]{dugger_motivic_cell}.
	\end{remark}
	
	\subsection{Base Change}
		Let $K$ be a field and $F$ a field extension of $K$. The base change functor on smooth schemes $f^*:\mathsf{Sm_K}\ra \mathsf{Sm_F}$ induces a symmetric monoidal base change functor $f^*:\SH(K)\ra \SH(F)$. The spectra we have defined behave well with respect to base change. Denote by $\underline{K}^{MW}_F$ the Milnor-Witt homotopy module in $\SH(F)$ and similarly for the Milnor, Witt and Milnor mod $2$ homotopy modules. 

	\begin{proposition}\label{pullback_preservesnicespectra}
		Let $k$ be a perfect field and $F$ a field extension. Then ${f^*}:\SH(k)\ra\SH(F)$ preserves $\underline{K}^{MW}, \underline{K}^W,\underline{K}^M,\underline{k}^M$ and also their effective covers, i.e. we have a natural equivalence $f^*(\underline{K}_k^R)\simeq \underline{K}_F^R$.
	\end{proposition}
	\begin{proof}
		The base change functor $f^*$ preserves homotopy colimits and truncations \cite[Lemma 2.2]{hopkins_morel}. Thus, it preserves the sphere spectrum and its $0$-truncation $\underline{K}^{MW}$. Furthermore, base change $f^*$ commutes with effective covers and slices \cite[Lemma 2.4]{slices_of_hermitian_k-theory}. Hence, it preserves $s_0(\Ss) \simeq H\Z$ \cite[Thm. 9.0.3 and 10.5.1]{levine_htpy_coniveau}\cite[Rem. 4.20]{hopkins_morel} and its $0$-truncation $\underline{K}^M$. Base change also preserves $H\Z/2$ and by the cofiber sequence 
		$$\Sigma^{0,-1}H\Z/2\xra{\tau} H\Z/2\ra \underline{k}^M$$
		it further preserves $\underline{k}^M$. By \cite[Thm. 5.3]{morel_sur_les_puissances_de_I} we have a cofiber sequence 
		$$\underline{K}^{MW}\ra \underline{K}^M\vee \underline{K}^W\ra \underline{k}^M.$$ We have shown that $f^*$ preserves all terms but $\underline{K}^W$; which it also must preserve by the cofiber sequence.
	\end{proof}
	
	\subsection{Completion}
		Let $E$ be a ring spectrum, $x\in \pi_0(E)_*$ and $M$ a left $E$-module spectrum. The \emph{$x$-completion} $E^\wedge_x$  of $E$ is defined by the homotopy limit 
		$$M^\wedge_x:=\lim_i M/x^i,$$
		where $M/x^i$ is the cofiber of the map $x^i:\Sigma^{i|x|} M\ra M$. $M$ is called \emph{$x$-complete} if the canonical map induces an equivalence $M\simeq M^\wedge_x$. The following theorem by Lurie shows that completion behaves well with respect to taking homotopy groups.
		\begin{theorem}\cite[XII, Thm. 4.2.13]{lurie_DAG}\label{completion_and_htpy_groups}
			Let $E\in \SH(k)$ be a connective $E_2$-ring spectrum, $M$ a left $E$-module and $x\in \pi_0(E)_*$. If $M$ is $x$-complete, then $\pi_k(M)_*$ is $x$-complete as a $\pi_0(E)_*$-module for every $k\in \Z$.
		\end{theorem}
		\begin{proof}
			This follows from \cite[XII, Prop. 4.2.16]{lurie_DAG}. The arguments in loc. cit. analogously apply for $\SH(k)$ with the homotopy $t$-structure.
		\end{proof}
		We will be most concerned with completing with respect to $\eta\in \pi_0(\Ss)_{-1}$. In this case, Bachmann and Hopkins have established a connection to $2$-completion for fields of finite virtual cohomological $2$-dimension \cite[Thm. 5.1]{bachmann-etaperiodic-fields}. It implies the following Proposition.
		
		\begin{proposition}\label{etacompletionimplies2complationandvv}
			Let $E\in \SH(k)$ be an $H_W$-module, then $\eta$-completion factors through $2$-completion, i.e. the canonical map is an equivalence: $E^\wedge_\eta\xra{\simeq}E^\wedge_{\eta,2}$. If $\vcd_2(k)<\infty$, and $E$ is very $n$-effective for some $n\in \Z$, then we have canonical equivalences $$E^\wedge_\eta\xra{\simeq}E^\wedge_{\eta,2}\xla{\simeq} E^\wedge_2.$$
		\end{proposition}
		\begin{proof}
			Since the relation $\eta\rho=-2$ holds in ${H_W}_{**}$, we have $(2)\subset (\eta)\subset {H_W}_{**}$. It follows that $\eta$-completion factors through $2$-completion; see for example \cite[XII, Prop. 4.2.11 and 4.2.12]{lurie_DAG}. The second claim is \cite[Thm. 5.1]{bachmann-etaperiodic-fields}.
		\end{proof}
		This proposition reduces $\eta$-completion to $2$-completion for $H_W$-modules. When the field $k$ is \emph{nonreal}, i.e. $-1$ is a sum of squares, then this even frees us from completing.
		\begin{corollary}\label{nonrealfieldsdontneedcompletion}
			Let the base field $k$ be nonreal of characteristic not $2$ with $\vcd_2(k)< \infty$. Then very $n$-effective $H_W$-modules are $\eta$-complete.
		\end{corollary}
		\begin{proof}
			When $k$ is nonreal, then $W(k)=\pi_0(H_W)_0$ is $2$-primary torsion of finite exponent \cite[31.4(6)]{quadratic_forms_kniga}.  Hence,  every $H_W$-module is $2$-complete. The claim follows by Proposition \ref{etacompletionimplies2complationandvv}.
		\end{proof}

%% file: motivic_steenrod.tex
\section{The Motivic Steenrod Algebra and its Dual}\label{sec_steenrod}
	The structure of ${H_W\Z}_{**}H_W\Z$ is deeply intertwined with the structure of the motivic Steenrod algebra mod 2 and its dual, $\mathcal{A}_2^{**}\cong H\Z/2^{**}H\Z/2$ and $\mathcal{A}^2_{**}\cong H\Z/2_{**}H\Z/2$ respectively. Since their properties will be used extensively in what follows, we decided to dedicate a new section to this. We collect some important facts about them here.
		
	\subsection{Self Homology and Cohomology of  Spectra}
		Let $E$ be a commutative ring spectrum in $SH(k)$ with multiplication $\mu:E\wedge E\ra E$ and unit $\unit:\Ss\ra E$. Then $E_{**}$ is a ring via $x\otimes y\mapsto \mu^*(x\wedge y)$ and both $E_{**}E$ and $E^{**}E$ are $E_{**}$-algebras. Their ring structures are defined as follows. Let $\mu_{i,j}:E^{\wedge 4}\ra E^{\wedge 3}$ be the multiplication map acting on the $i$-th and $j$-th coordinate and $\iota_{i,j}:E^{\wedge n}\ra E^{\wedge n}$ the map swapping coordinates $i$ and $j$. The multiplication on $E_{**}E$ is given by $x\otimes y\mapsto (\mu_{1,2}\wedge \mu_{3,4}) \circ \iota_{2,3}\circ (x\wedge y)$. It is graded commutative. The multiplication on $E^{**}E$ is given by composition $f\otimes g \mapsto f\circ g$. This is in general not commutative. We denote by $\iota:E\wedge E \ra E\wedge E$ the conjugation map. $E^{**}E$ is an $E_{**}$-algebra, and hence a bimodule, via the ring map 
		$$E_{**} \ra E^{**}E, \hspace{1cm} s\mapsto E\simeq \Ss \wedge E \xra{s\wedge E} E\wedge E \xra{\mu} E.$$
		$E_{**}E$ gets its $E_{**}$-algebra structure via the ring map $$E_{**} \xra{\eta_L} E_{**}E, \hspace{1cm} s\mapsto s\wedge \unit .$$
		There is a second ring map, $\eta_R$, which defines a right $E_{**}$-module structure on $E_{**}$ given by $s\mapsto \unit \wedge s$.
		\begin{remark}
			The image of $E_{**}$ does generally not lie in the center of $E^{**}E$. For example, when $E$ is motivic cohomology mod $2$, then $Sq^1\cdot \tau - \tau \cdot Sq^1=\rho \neq 0$. Classical examples of this phenomenon are also $MU$ and $BP$. 
		\end{remark}
		Recall the exterior products for $X, Y\in SH(k)$ 
		\begin{align*}
			E_{**}(X)\otimes E_{**}(Y)& \xra{\times} E_{**}(X\wedge Y), \hspace{1cm} x\otimes y\mapsto (\mu_{1,3}\wedge X\wedge Y)\circ (x\wedge y),\\
			E^{**}(X)\otimes E^{**}(Y)& \xra{\times} E^{**}(X\wedge Y), \hspace{1cm} f\otimes g \mapsto \mu \circ (f\wedge g),
		\end{align*}
		as well as the Kronecker pairing 
		$$E^{-*-*}(X)\otimes E_{**}(X)\xra{\langle -, - \rangle} E_{**}, \hspace{1cm}f\otimes x \mapsto \langle f,x \rangle:= \mu \circ (E\wedge f)\circ x.$$
		Here the degrees are additive, i.e. $|\langle f,x \rangle| = |x|-|f|$.  Given a map $\phi:X\ra Y$, we have $\langle f,\phi_*(x)\rangle=\langle \phi^*(f),x\rangle$. The pairings are meant to be understood to be over $X$ and $Y$ respectively. Also note that in this language we have $x\cdot y=\mu_*(x\times y)$ for $x,y\in E_{**}E$.
		\begin{proposition}
			Let $\alpha, f\in E^{**}E$, $x\in E_{**}E$, $s\in E_{**}$. Then we have 
			\begin{enumerate}[(i)]
				\item $\alpha_*$ is dual to $\alpha^*$: $\langle f,\alpha_*(x) \rangle=\langle \alpha^*(f),x\rangle = \langle f\cdot \alpha, x \rangle$,
				\item $\langle f, \alpha_*(x)\rangle = \langle \alpha, f_* (\iota(x)) \rangle$,
				\item the pairing is left $E_{**}$-linear in both variables: $s \langle f,x \rangle = \langle sf,x \rangle= \langle f, sx \rangle$, 
				\item the pairing respects the right $E_{**}$-actions: $\langle fs,x \rangle = \langle f,xs \rangle$.
			\end{enumerate}
		\end{proposition}
		\begin{proof}
			See \cite{symmspectra_schwede} or \cite{Voevodsky_reduced_power_operations_in_motivic_cohomology} specifically for the case of the motivic Steenrod algebra.
		\end{proof}
	
		\begin{remark}
			\begin{enumerate}[(i)]
				\item The pairing does not necessarily commute with right actions. Let $E=H\Z/2$ be motivic cohomology mod $2$ and let $Sq^1\in H\Z/2^{**}H\Z/2$ be the Bockstein. Using the description (\ref{dual_steenrod}), we have $\langle Sq^1, 1\cdot \tau \rangle = \langle Sq^1, \tau +\rho \tau_0 \rangle =\rho\neq \tau \langle Sq^1, 1 \rangle =0$, since $Sq^1$ is dual to $\tau_0$. 
				\item $c$ swaps the left and right action on $E_{**}$. When they disagree, then $\iota$ is not $E_{**}$-linear and hence cannot be dualized via the pairing, in case it is perfect. In other words, there is no map $\iota:E^{**}E \ra E^{**}E$ satisfying $\langle f,\iota(x)\rangle=\langle \iota(f), x \rangle$. The existence of a dual conjugate forces the pairing to be right $E_{**}$-linear in both variables which on the other hand forces the left and right actions to be identical in the perfect case.
				\item The last comment suggests that, even if $(E_{**},E_{**}E)$ is a Hopf algebroid, then $(E^{**},E^{**}E)$ might not be a Hopf algebroid in any meaningful way. This is indeed the case. For $H\Z/2$, the motivic Steenrod algebra mod 2, $(H\Z/2^{**},H\Z/2^{**}H\Z/2)$ is not a graded Hopf algebroid over fields not containing a square root of $-1$. One can show that there cannot be a conjugation on $H\Z/2^{**}H\Z/2$ which preserves the embedding $H\Z/2^{**}\ra H\Z/2^{**}H\Z/2$. 
			\end{enumerate} \label{hopf_algebroids_remark}
		\end{remark}

		The multiplication map $\mu:E\wedge E \ra E$ defines a diagonal $\mu^*:E^{**}E \ra E^{**}(E\wedge E)$. For $e,h\in E^{**}(X\wedge Y), a,b\in E_{**}(X\wedge Y)$ we have \cite[Prop. 6.17 (iii)]{symmspectra_schwede} 
		$$\langle e\times h, a\times b\rangle =(-1)^{tdim(g)\cdot tdim(a)}\langle e,a\rangle \cdot \langle h,b\rangle.$$ If we have a decomposition $\mu^*(f)=\sum_i f_i'\times f_i''$, since the diagonal is dual to multiplication on $E_{**}E$, the above formula implies: $$\langle f, x\cdot y\rangle=\sum_i (-1)^{tdim(f_i'')tdim(x)}\langle f_i',x\rangle \cdot \langle f_i'',y\rangle,$$
		where we define $tdim(a)=p$ for $a\in E_{p,q}(X\wedge Y)$ and $tdim(e)=i$ for $e\in E^{i,j}(X\wedge Y)$.	
		Similarly, one defines a diagonal $$\psi:E_{**}E\ra E_{**}(E\wedge E), \hspace{1cm} x\mapsto \iota_{2,3}\circ (x\wedge \unit),$$
		with $\iota_{i,j}:E^{\wedge n}\ra E^{\wedge n}$ swapping coordinates $i$ and $j$.
	
	\subsection{The (Dual) Motivic Steenrod Algebra }
		We now specialize to $E=H\Z/2$, i.e. motivic cohomology mod 2. In this case $H\Z/2\wedge H\Z/2$ is a reflexive $H\Z/2$-module, i.e. the canonical map $$H\Z/2\wedge H\Z/2 \ra \Hom_{H\Z/2}(\Hom_{H\Z/2}(H\Z/2\wedge H\Z/2, H\Z/2),H\Z/2)$$ is an equivalence of $H\Z/2$-modules \cite[Prop. 5.3]{hoyois-motivic_steenrod_algebra}\cite[Prop. 5.5, Thm. 5.10]{hopkins_morel}. Thus, the Kronecker pairing is perfect. The diagonal on $H\Z/2^{**}H\Z/2$ can now be identified with 
		$$\Delta:H\Z/2^{**}\xra{\mu^*} H\Z/2^{**}(H\Z/2\wedge H\Z/2)\cong H\Z/2^{**}H\Z/2\otimes_{H\Z/2^{**}}H\Z/2^{**}H\Z/2.$$
		The decomposition $\mu^*(f)=\sum_i f_i'\times f_i''=\sum_i f_i'\otimes f_i''$ now always holds, where we identified via above isomorphism. Hence, we always have 
		$$\langle f, x\cdot y\rangle=\sum_i \langle f_i',x\rangle \cdot \langle f_i'',y\rangle.$$
		Dually, we have a diagonal 
		$$\Delta:H\Z/2_{**}\xra{\psi} H\Z/2_{**}(H\Z/2\wedge H\Z/2)\cong H\Z/2_{**}H\Z/2\otimes_{H\Z/2_{**}}H\Z/2_{**}H\Z/2$$
		and for $\Delta(x)=\sum_i x_i' \otimes x_i''$ we get \cite[12.9]{Voevodsky_reduced_power_operations_in_motivic_cohomology}\cite[Lemma 3.3]{boardmaneightfoldway} 
		$$\langle f\circ g, x \rangle = \sum_i \langle f\langle g,x_i'' \rangle,x_i' \rangle,$$
		where the scalar $\langle g,x_i''\rangle$ acts on the right on $f$, i.e. $f\langle g,x_i''\rangle=\eta_R(\langle g,x_i''\rangle)\cdot f$.

		Recall from Section \ref{sec_eff_and_veff} that we know the coefficient ring $H\Z/2_{**}$ and that it is given by
		$$H\Z/2_{**}\cong \underline{k}^M_{**}[\tau],$$
		i.e. $H\Z/2_{p,q}\cong k_p^M(k)\tau^{q-p}$ for $p\geq q$ and $0$ otherwise. Note that this is a commutative ring. The elements $k^M_*(k)$ all come from the sphere spectrum while $\tau$ does not. Recall that we denote the element $\Ss \ra \Gm\simeq \Sigma^{1,1}\Ss$ corresponding to $-1\in \Gm(k)$ by the letter $\rho$.
	
		The dual motivic Steenrod algebra is given as a commutative left $H\Z/2_{**}$-algebra by \cite{Voevodsky_reduced_power_operations_in_motivic_cohomology}\cite{hopkins_morel} \cite{hoyois-motivic_steenrod_algebra}
		\begin{equation}\label{dual_steenrod}
			H\Z/2_{**}H\Z/2\cong H\Z/2_{**}[\tau_0,\tau_1,\ldots,\xi_1,\xi_2,\ldots]/(\tau_i^2-\rho\tau_{i+1}-(\tau+\rho\tau_0)\xi_{i+1})
		\end{equation}
		with degrees 
		$$|\tau_i|=(2^{i+1}-1,2^i-1),\hspace{1cm} |\xi_r|=(2^{i+1}-2,2^i-1).$$ 
		\begin{remark}
			We would like to point out that there is no immediate connection between the element $\tau$ in the coefficient ring and the elements $\tau_i$ in the dual Steenrod algebra, contrary to what the notation might suggest.
		\end{remark}
		We readily see that $H\Z/2_{**}H\Z/2$ is free as a left $H\Z/2_{**}$-module. The basis can be described in terms of monomials. Let $E$ and $R$ be finite length sequences $E=(E_0,E_1,\ldots)$, $R=(R_1,R_2,\ldots)$ with $E_i\in \{0,1\}$ and $R_i\in \N_0$. Denote by $\emptyset$ the zero-sequence and by $e_i$ the sequence with a single $1$ at the $i$-th position. Basis monomials can be given by $$\tau(E)\xi(R)=\tau_0^{E_0}\tau_1^{E_1}\cdots \xi_1^{R_1}\xi_2^{R_2}\cdots.$$ The pair $(H\Z/2_{**},H\Z/2_{**}H\Z/2)$ is a Hopf algebroid. Its structure is described in \cite[Sec. 12]{Voevodsky_reduced_power_operations_in_motivic_cohomology}\cite[Sec. 5.1]{hoyois-motivic_steenrod_algebra}\cite[Sec. 5.3]{hopkins_morel}. We recall
		\begin{align}\label{hopf_algebroid_structure}
			\begin{split}
				\eta_L(\tau)&=\tau,\\
				\eta_R(\tau)&=\tau+\rho\tau_0,\\
				\Delta(\tau)&=\tau\otimes 1,\\
				\Delta(\tau_r)&=\tau_r\otimes 1 +1\otimes \tau_r +\sum_{i=0}^{r-1}\xi_{r-i}^{2^i}\otimes \tau_i,\\
				\Delta(\xi_r)&=\xi_r\otimes 1 +1\otimes \xi_r +\sum_{i=1}^{r-1}\xi_{r-i}^{2^i}\otimes \xi_i,\\
				\iota(\tau)&=\tau+\rho\tau_0,\\
				\iota(\tau_r)&=\tau_r+\sum_{i=0}^{r-1}\xi_{r-i}^{2^i}\iota(\tau_i),\\
				\iota(\xi_r)&=\xi_r+\sum_{i=1}^{r-1}\xi_{r-i}^{2^i}\iota(\xi_i).
			\end{split}
		\end{align} 
		We note that $\eta_L$, $\eta_R$, $\Delta$ and $c$ are $k^M_*(k)$-linear, since these elements all come from the sphere spectrum. Whenever we write $x\cdot s=xs$ for a scalar $s\in H\Z/2_{**}$, we will mean $\eta_R(s)x$. In particular, $x\tau=\eta_R(\tau)x=(\tau+\rho\tau_0)x$. Moreover, we will denote the following conjugates with overlines $\iota(\tau_i)=\overline{\tau_i}$, $\iota(\xi_i)=\overline{\xi_i}$. 
	
		The motivic Steenrod algebra $H\Z/2^{**}H\Z/2$ is generated as an algebra by $H\Z/2^{**}$ and reduced power operations $$Sq^{2i}\in H\Z/2^{2i,i}H\Z/2,\hspace{1cm} Sq^{2i+1}\in H\Z/2^{2i+1,i}H\Z/2,$$ 
		where $i\geq 0$, $Sq^0=id$ and $Sq^1$ is the Bockstein.
		Since the Kronecker pairing is perfect, we can dualize the basis $\tau(E)\xi(R)$. The dual of $\tau_0$ is $Sq^1$ and the dual of $\xi_1$ is $Sq^2$. Our computations will involve the maps 
		$$H\Z/2_{**}H\Z/2\xra{(Sq^2)^R_*}H\Z/2_{**}H\Z/2 \hspace{0.5cm} \text{ and }\hspace{0.5cm} H\Z/2_{**}H\Z/2\xra{(Sq^2)_*^L}H\Z/2_{**}H\Z/2$$ 
		given by $\pi_{**}(H\Z/2\wedge Sq^2)$ and $\pi_{**}(Sq^2\wedge H\Z/2)$ respectively. They can be computed as follows.
		\begin{proposition}\label{*alphaandalpha*formula}
			Let $x\in H\Z/2_{**}H\Z/2$ and $\alpha \in H\Z/2^{**}H\Z/2$. Then 
			$$\alpha^R_*(x)=\sum_i x_i' \langle \alpha, x_i''\rangle \hspace{1cm} \text{ and } \hspace{1cm} \alpha_*^L(x)=\sum_i \langle \alpha, \iota(x_i')\rangle x_i''. $$
		\end{proposition}
		\begin{proof}
			This can be shown using formulas for the Kronecker pairing that we have established earlier. The second formula follows from Corollary 6.4 in \cite{boardmaneightfoldway}.
		\end{proof}
		\begin{remark}
			Over a field with no square root of $-1$, the conjugation in the second factor is necessary even for $\alpha=Sq^1$, the dual element of $\tau_0$: 
			\begin{gather*}
			\langle Sq^1,\tau\rangle=0,\\
			\langle Sq^1, \iota(\tau)\rangle=\langle Sq^1, (\tau+\rho\tau_0)\rangle=\rho\neq 0.
			\end{gather*}
			Classically, we can dualize the conjugation and the conjugate of $Sq^1$ is $Sq^1$ itself. Hence, this does not occur classically and our reasoning suggests that this is related to Remark \ref{hopf_algebroids_remark} (iii).
		\end{remark}
		\begin{remark}
			We also have $\iota\circ (Sq^2)^R_*=(Sq^2)_*^L\circ \iota$.
		\end{remark}
		Their behavior on products is described by
		\begin{proposition}\label{Sq2_on_products}
			We have for $x,y\in H\Z/2_{**}H\Z/2$ 
			\begin{align*}
				(Sq^1)^R_* (x\cdot y)&=(Sq^1)^R_* (x)\cdot y+x\cdot (Sq^1)^R_*(y),\\
				(Sq^1)_*^L(x\cdot y)&=(Sq^1)_*^L(x)\cdot y+x\cdot(Sq^1)_*^L(y),\\
				(Sq^2)^R_* (x\cdot y)&=(Sq^2)^R_* (x)\cdot y+x\cdot (Sq^2)^R_*(y)+(\tau+\rho\tau_0)\cdot (Sq^1)^R_*(x)\cdot (Sq^1)^R_*(y),\\
				(Sq^2)_*^L(x\cdot y)&=(Sq^2)_*^L(x)\cdot y+x\cdot(Sq^2)_*^L(y)+\tau\cdot (Sq^1)^L_*(x)\cdot (Sq^1)^L_*(y).
			\end{align*}
		\end{proposition}
		\begin{proof}
			By the Cartan formula \cite[Prop. 9.7, Lemma 11.6]{Voevodsky_reduced_power_operations_in_motivic_cohomology}) and the fact that the Bockstein $Sq^1$ \cite[(8.1)]{Voevodsky_reduced_power_operations_in_motivic_cohomology} is a derivation, we have 
			\begin{align*}
				\Delta(Sq^1)&=Sq^1\otimes 1+1\otimes Sq^1,\\
				\Delta(Sq^2)&=Sq^2\otimes 1+1\otimes Sq^2 +\tau Sq^1\otimes Sq^1.
			\end{align*} 
			The result now follows from duality of $\Delta$ and multiplication on $H\Z/2_{**}H\Z/2$ together with Proposition \ref{*alphaandalpha*formula}.
		\end{proof}

		The left $H\Z/2_{**}$-subalgebra generated by $\xi_i$ is preserved under conjugation. Hence 
		\begin{align*}
			H\Z/2_{**}H\Z/2&\cong H\Z/2_{**}[\tau_0,\tau_1,\ldots,\xi_1,\xi_2,\ldots]/(\tau_i^2=\rho\tau_{i+1}+\xi_{i+1}\tau)\\
			&\cong H\Z/2_{**}[\tau_0,\tau_1,\ldots,\overline{\xi_1},\overline{\xi_2},\ldots](\tau_i^2=\rho\tau_{i+1}+\xi_{i+1}\tau).
		\end{align*}
		Let $K\subset H\Z/2_{**}H\Z/2$ be the left $H\Z/2_{**}$-module spanned by monomials 
		$$\tau_0^{E_0}\tau_1^{E_1}\cdots\overline{\xi_1}^{2R_1}\overline{\xi_2}^{R_2}\overline{\xi_3}^{R_3}\cdots,$$ with $E_i\in \{0,1\}$ and $R_i\in \N_0$. Then we have an obvious splitting
		\begin{equation}\label{H**Hspliting}
			H\Z/2_{**}H\Z/2\cong K\oplus \overline{\xi_1}K.
		\end{equation}
		We will use this splitting later to compute $H\Z/2$-homology of $H_W\Z$. The elements in $K$ map to zero under $(Sq^2)^R_*$ modulo $(\tau+\rho\tau_0)$.

%% file: h_w-motives.tex
\section{Building Blocks of \texorpdfstring{$H_W\Z\wedge H_W\Z$}{the Dual Motivic Witt Steenrod algebra}}
	As mentioned in the introduction, $H_W\Z$ and $H\Z/2$ share an intimate relation with each other witnessed by the pullback square of rings \cite[Thm 5.3]{morel_sur_les_puissances_de_I}\cite[Chapter 5]{morela1algtopoverfield}
	\begin{equation}
		\begin{tikzcd}
			K_*^{MW}(k) \ar[r] \ar[d]& K_*^M(k) \ar[d]\\
			K_*^W(k) \ar[r] & k_*^M(k).
		\end{tikzcd}
	\end{equation}
	We will make precise the details of their connection now.  We then use this information to break down the task of computing $H_W\Z_{**} H_W\Z$ into pieces which we compute in the Sections after.	
	
	Theorem 17 of \cite{bachmann_generalized-slices} demonstrates that we have a diagram of cofiber sequences
	\begin{equation}\label{grid_H_HW}\begin{tikzcd}
		* \arrow[r] \arrow[d] & \Sigma^{0,-1} H\Z/2 \arrow[r, "id"] \arrow[d, "\tilde{\tau}"] & \Sigma^{0,-1} H\Z/2  \arrow[d, "\tau"]   \\
		\Sigma^{1,1} \underline{K}^W  \arrow[d, "id"] \arrow[r,"\tilde{\eta}"] & H_W\Z \arrow[r,"r"] \arrow[d, "\pi_W"] & H\Z/2  \arrow[d, "\pi"]  \\
		\Sigma^{1,1} \underline{K}^W \arrow[r,"\eta"] & \underline{K}^W \arrow[r,"\underline{r}"] & \underline{k}^M .
	\end{tikzcd}\end{equation}
	We will denote the maps continuing to the right of $r$ and $\underline{r}$ by $\partial$ and $\underline{\partial}$ respectively. Here the middle row is the effective cover $f_0$ of the bottom row which is the cofiber sequence (\ref{cofiber_kW_k}).
	The two bottom vertical maps $\pi_W$ and $\pi$ are given by zero truncation with respect to the homotopy t-structure. That is, the columns are of the form $X_{\geq 1}\ra X \ra X_{\leq 0}$. 
	\begin{notation}\label{notation_underline_tilde}
		Underlined maps will usually mean $0$-truncated maps with respect to the homotopy t-structure. In that case we have a relationship such as $\pi \circ r=\underline{r}\circ \pi_W$ in (\ref{grid_H_HW}). A tilde map in most cases is a lifted version of a map. For example: $\eta=\pi_W\circ \tilde{\eta}$ in (\ref{grid_H_HW}).
	\end{notation}
	Diagram (\ref{grid_H_HW}) immediately implies the following Proposition.
	\begin{proposition}\label{fundamental-pullback-square}
		We have a homotopy pullback square
		\[\begin{tikzcd}
		H_W\Z \ar[r,"r"] \ar[d,"\pi_W"] & H\Z/2 \ar[d,"\pi"]\\
		\underline{K}^W \ar[r, "\underline{r}"] & \underline{k}^M.
		\end{tikzcd}\]
	\end{proposition}
	We smash the pullback square in Proposition \ref{fundamental-pullback-square} with $H_W\Z$ on the right and get a pullback square 
	\begin{equation}\label{H_Wwedgefundamental_pullbacksquare}
		\begin{tikzcd}
			H_W\Z\wedge H_W\Z \ar[r,"r\wedge H_W\Z"] \ar[d,"\pi_W\wedge H_W\Z"] & H\Z/2\wedge H_W\Z \ar[d,"\pi\wedge H_W\Z"]\\
			\underline{K}^W\wedge H_W\Z \ar[r,"\underline{r}\wedge H_W\Z"] & \underline{k}^M\wedge H_W\Z.
		\end{tikzcd}
	\end{equation}
	In the Sections we will determine the algebra structure of the homotopy groups of the pieces in this square surrounding $H_W\Z\wedge H_W\Z$. We proceed clockwise and begin with the top right corner, $H\Z/2_{**}H_W\Z$.
	
	\subsection{Computation of \texorpdfstring{$H\Z/2\wedge H_W\Z$}{ mod 2 motivic Homology of Hw}}\label{sec_H_H_W}
		We start our computations with $H\Z/2\wedge H_W\Z$. Smashing the pullback square in Proposition \ref{fundamental-pullback-square} with $H\Z/2$ on the left, we have a pullback square 
		\[\begin{tikzcd}
			H\Z/2\wedge H_W\Z \ar[r,"H\Z/2\wedge r"] \ar[d,"H\Z/2\wedge \pi_W"] & H\Z/2\wedge H\Z/2 \ar[d,"H\Z/2\wedge \pi"]\\
			H\Z/2\wedge \underline{K}^W \ar[r,"H\Z/2\wedge \underline{r}"] & H\Z/2\wedge \underline{k}^M.
		\end{tikzcd}\]
		We know the homotopy groups on the right. Hence, to compute $H\Z/2_{**}H_W\Z$ we are only missing the structure of $H\Z/2\wedge \underline{K}^W$.
		
		The map $\eta:\Ss^{1,1}\ra\Ss$ is a map between spheres and acts as $0$ on $H\Z$, and thus also on $H\Z/2$. Therefore, when smashing  the last row of grid (\ref{grid_H_HW}) with $H\Z/2$ on the left all rows are split. That is, they are of the form $X\ra X\vee F\ra F$. Consequently, we have a split short exact sequence 
		\begin{equation}{\label{ses_hkw}}
			0 \ra H\Z/2_{**}\underline{K}^W \xra{\underline{r}^R_*} H\Z/2_{**}\underline{k}^M \xra{\underline{\partial}^R_*} \Sigma^{2,1}H\Z/2_{**}\underline{K}^W\ra 0.
		\end{equation}
		Let $d:=\underline{r}\circ \underline{\partial}$. The $H\Z/2_{**}$-homology of $\underline{K}^W$ can now be computed as the kernel of $d^R_*$. The following Proposition enables us to do so by giving a description of $d$. This proposition can already be found in the proof of \cite[Lemma 6.18]{bachmann-etaperiodic-fields} which borrow below.
		\begin{proposition}\label{description_of_d}
			We have a commutative diagram 
			\[\begin{tikzcd}
			H\Z/2 \ar[d,"\pi"]\ar[r,"Sq^2"] &\Sigma^{2,1} H\Z/2 \ar[d,"\pi"]\\
			\underline{k}^M\ar[r,"d"] & \Sigma^{2,1} \underline{k}^M.
			\end{tikzcd}\]
		\end{proposition}
		\begin{proof}
			We proceed as in the proof of \cite[Lemma 6.18]{bachmann-etaperiodic-fields}. Taking $[H\Z/2,-]$ of the $\Sigma^{2,1}$-shifted version of the cofiber sequence (\ref{tau_cofiber_seq}), we see that we have a surjection $$\pi_*:H\Z/2^{2,1}H\Z/2 \thra [H\Z/2,\Sigma^{2,1}\underline{k}^M],\hspace{0.5cm} f\mapsto \pi\circ f,$$
			since $\tau$ is a non-zero divisor in the motivic Steenrod algebra.
			Hence, there exists $f\in H\Z/2^{2,1}H\Z/2$ such that $d\circ \pi=\pi\circ f$. Taking $H\Z/2$-homology, we get a commutative diagram
			\[\begin{tikzcd}
			H\Z/2_{**}H\Z/2 \ar[d,twoheadrightarrow,"\pi^R_*"]\ar[r,"f^R_*"] &\Sigma^{2,1} H\Z/2_{**}H\Z/2 \ar[d,twoheadrightarrow,"\pi_*^R"]\\
			H\Z/2_{**}\underline{k}^M\ar[r,"d^R_*"] & \Sigma^{2,1} H\Z/2_{**}\underline{k}^M,
			\end{tikzcd}\]
			where we used the rightmost column of (\ref{grid_H_HW}) and the fact that $\tau$ is a non-zero divisor in the dual motivic Steenrod algebra to deduce that the vertical maps are surjective. By degree reasons we have $H\Z/2^{2,1}H\Z/2=\{0,\rho Sq^1,Sq^2,Sq^2+\rho Sq^1\}$. Using again the rightmost column of (\ref{grid_H_HW}), we see $H\Z/2_{**}\underline{k}^M\cong H\Z/2_{**}H\Z/2/\tau$. Since this is never trivial and the left and right side of (\ref{ses_hkw}) are shifts of each other, the splitting is never trivial. It follows that $d\neq 0$. After base change to a quadratically closed field we have $\rho=0$, hence we cannot have $f=\rho Sq^1$. We note that the short exact sequence (\ref{ses_hkw}) dictates $d^R_*\circ d^R_*=0$. Using the formulas (\ref{hopf_algebroid_structure}) for $\Delta$ we see that $((Sq^2+\rho Sq^1)^R_*)^2(\tau_1)=(Sq^2+\rho Sq^1)^R_*(\rho\tau_0)=\rho$ is not zero under the map $\pi_*^R$. We must have $f=Sq^2$. This shows the claim.
		\end{proof}
		\begin{notation}
			In concordance with our convetions in Notation \ref{notation_underline_tilde}, we will also refer to $d$ by $\underline{Sq^2}$. 
		\end{notation}
		
		Before we begin our computation, let us remark how this computes the algebra $H\Z/2_{**} H_W\Z$. 
		\begin{lemma}\label{H**HW_square}
			We have a pullback square of rings
			\[\begin{tikzcd}
			{H\Z/2}_{**}H_W\Z \ar[r,hook,"r^R_*"] \ar[d,twoheadrightarrow,"(\pi_W)^R_*"] & H\Z/2_{**}H\Z/2 \ar[d,twoheadrightarrow,"\pi^R_*"]\\
			H\Z/2_{**}\underline{K}^W_{**} \ar[r,hook,"\underline{r}^R_*"] & H\Z/2_{**}\underline{k}^M.
			\end{tikzcd}\]
		\end{lemma}
		\begin{proof}
			As $(\ref{ses_hkw})$ shows, $\underline{r}^R_*$ is an injection on $H\Z/2_{**}$-homology, hence $r^R_*$ injects $H\Z/2_{**}H_W\Z$ into $H\Z/2_{**}H\Z/2$ by a simple diagram chase. The map $\tau^R_*$, which is just given by multiplication by $\tau$ on the right, is injective since $\tau$, and therefore $\iota(\tau)$, is a non-zero divisor in $H\Z/2_{**}H\Z/2$. Hence, $\pi^R_*$ and thus $(\pi_W)^R_*$ are surjective. The Mayer-Vietoris sequence shows that if one of the maps one pulls back along is surjective on homotopy groups, then one has a pullback square on homotopy groups. This and the fact that all involved maps are maps of ring spectra now imply the claim.
		\end{proof}
		
		To compute the kernel of $d^R_*$, i.e. $H\Z/2_{**}\underline{K}^W$, we use its behavior on products. Recall Proposition \ref{Sq2_on_products}: 
		$$(Sq^2)^R_* (x\cdot y)=(Sq^2)^R_* (x)\cdot y+x\cdot (Sq^2)^R_*(y)+(\tau+\rho\tau_0)\cdot (Sq^1)^R_*(x)\cdot (Sq^1)^R_*(y).$$ 
		This means $d^R_*$ is a derivation, since it is obtained from $(Sq^2)^R_*$ by modding out by $(\tau+\rho\tau_0)$ as can be seen by Corollary \ref{description_of_d}. The formulas (\ref{hopf_algebroid_structure}), Proposition $\ref{*alphaandalpha*formula}$ and Proposition \ref{Sq2_on_products} imply inductively
		$$(Sq^2)^R_*(\tau_i)=0, \hspace{1cm} (Sq^2)^R_*(\overline{\xi_1})=1,\hspace{1cm}(Sq^2)^R_*(\overline{\xi_j})=0,\hspace{0.5cm}\text{for }j\neq 1.$$
		\begin{proposition}\label{H_KW}
			We have an isomorphism of left $H\Z/2_{**}$-algebras 
			$$ H\Z/2_{**} \underline{K}^W \overset{\underline{r}^R_*}{\underset{\cong}\lra}H\Z/2_{**}[\tau_0,\ldots, \overline{\xi_1}^2,\overline{\xi_2},\overline{\xi_3}\ldots]/(\tau+\rho\tau_0,\tau_i^2=\rho\tau_{i+1}) \hra  H\Z/2_{**}\underline{k}^M.$$
		\end{proposition}
		\begin{proof}
			Let $K$ be as in (\ref{H**Hspliting}). By what we have just computed, $\underline{K}:=\pi_*^R(K)$ lies in the kernel of $d^R_*$. Furthermore, $d^R_*$ acts on $\overline{\xi_1} \underline{K}$ via $\overline{\xi_1} k\mapsto k$. So $\im(\underline{r}^R_*)=\underline{K}$ and since $\underline{r}^R_*$ is a ring map, $\underline{K}$ is a left $H\Z/2_{**}$-algebra.
			
			To deduce the relations, note that monomials
			$$\tau_0^{E_0}\tau_1^{E_1}\cdots\overline{\xi_1}^{R_1}\overline{\xi_2}^{R_2}\overline{\xi_3}^{R_3}\cdots,$$ with $E_i\in \{0,1\}$ and $R_i\in \N_0$, form a right $\underline{k}^M_{**}$-basis for $H\Z/2_{**}\underline{k}^M$. In particular, the monomials $$\tau_0^{E_0}\tau_1^{E_1}\cdots\overline{\xi_1}^{2R_1}\overline{\xi_2}^{R_2}\overline{\xi_3}^{R_3}\cdots$$ are therefore $k^M_{*}$-linearly independent. Any other relations in $H\Z/2_{**}\underline{K}^W$ must therefore contain higher degrees of $\tau_i$. Since the relations in the result obviously hold, we already have relations for $\tau_i^2$. Hence, we may reduce to lower degrees in them. The claim follows.
		\end{proof}
		\begin{remark}
			Since we had no assumptions on our base field, it follows that the right $k_*^M(k)$-basis of $H\Z/2_{**}\underline{K}^W$ mentioned in the proof describes $H\Z/2\wedge \underline{K}^W$ as a free right module over $\underline{k}^M$.
		\end{remark}
		
		\begin{proposition}\label{H_H_W}
			We have an isomorphism of left $H\Z/2_{**}$-algebras 
			$$ H\Z/2_{**} H_W\Z \overset{r^R_*}{\underset{\cong}\lra}H\Z/2_{**}[\tau_0,\ldots, \overline{\xi_1}\tau,\overline{\xi_1}^2,\overline{\xi_2},\overline{\xi_3}\ldots]/(\ldots) \hra  H\Z/2_{**}H\Z/2,$$
			where the ideal $(\ldots)$ is generated by the relations 
			$$\tau_i^2=\rho\tau_{i+1}+\xi_{i+1}\tau,\hspace{1cm} (\overline{\xi_1}\tau)^2=\overline{\xi_1}^2\tau^2.$$
		\end{proposition}
		\begin{proof}
			That we have the given generators follows from the pullback square of Lemma \ref{H**HW_square}. The first column of the grid (\ref{grid_H_HW}) gives a short exact sequence 
			$$0\ra \Sigma^{0,-1} H\Z/2_{**}H\Z/2\xra{\tilde{\tau}^R_*}H\Z/2_{**}H_W\Z\xra{(\pi_W)^R_*}H\Z/2_{**}\underline{K}^W\ra 0$$ 
			with $\tilde{\tau}^R_*(\xi_{i+1})= \tau_i^2-\rho\tau_{i+1}$ and $\tilde{\tau}^R_*(\tau_i)=(\tau+\rho\tau_0)\tau_i.$ 
			The formulas (\ref{hopf_algebroid_structure}) for $\overline{\xi_i}$ demonstrate that it suffices to have $\overline{\xi_1}\tau$ and $\overline{\xi_{j}}$, with $j\geq 2$, to recreate $\overline{\xi_{j}}\tau$. For more details see Remark \ref{elements_in_H_H_W}. The given set thus generates $H\Z/2_{**}H_W\Z$ as an algebra. The relations for $\tau_i^2$ are from $H\Z/2_{**}H\Z/2$. The second relation $(\overline{\xi_1}\tau)^2=\overline{\xi_1}^2\tau^2$ obviously holds in $H\Z/2_{**}H\Z/2$ and thus in $H\Z/2_{**}H_W\Z$. We have reduced all possible relations to lower powers of the generators. It remains to show that the monomials 
			$$\tau_0^{E_0}\tau_1^{E_1}\cdots(\overline{\xi_1}\tau)^\epsilon\overline{\xi_1}^{2R_1}\overline{\xi_2}^{R_2}\overline{\xi_3}^{R_3}\cdots,$$ with $\epsilon,E_i\in \{0,1\}$ and $R_i\in \N_0$, form a basis. In other words $H\Z/2_{**}H_W\Z\cong K\oplus K\overline{\xi_1}\tau$. To show that the above monomials are a basis It suffices to show that $K\oplus K\overline{\xi_1}\tau$ injects into $H\Z/2_{**}H\Z/2$ via the natural map $(a,b)\mapsto a+b$. Given the splitting (\ref{H**Hspliting}), it hence suffices to show that the map 
			\begin{equation}\label{temp2}
				K\overline{\xi_1}\tau\hra H\Z/2_{**}H\Z/2\cong K\oplus \overline{\xi_1} K\thra \overline{\xi_1}K
			\end{equation}
			is an injection. Multiplying a basis element $b$ with $\overline{\xi_1}\tau=\tau\overline{\xi_1}+\rho\tau_0\overline{\xi_1}$ will produce a term of the form $\tau\overline{\xi_1} b+T$ with $T$ some term with all summands being greater than $\overline{\xi_1}b$ in terms of a lexicographic order $\tau_i<\overline{\xi_i}$. Indeed, expanding $\tau_i^2$ we get $\rho\tau_{i+1}+\tau\xi_{i+1}+\rho\tau_0\xi_{i+1}$, a term that only leads to greater value with respect to the lexicographic order by the conjugation formulas $(\ref{hopf_algebroid_structure})$. Thus, the matrix of the map 
			\begin{equation}\label{temp3}
				K\hra H\Z/2_{**}H\Z/2\cong K\oplus \overline{\xi_1} K\xhra{\cdot \overline{\xi_1}\tau} K\oplus \overline{\xi_1} K \thra K\overline{\xi_1}{\tau}
			\end{equation}
			is lower triangular with $\tau$ on a diagonal shifted by the value of $\overline{\xi_1}$ with respect to the chosen lexicographic order. Hence (\ref{temp3}), and thus also (\ref{temp2}), are  injective. Therefore $H\Z/2_{**}H_W\Z\cong K\oplus K\overline{\xi_1}\tau$ as left $H\Z/2_{**}$-modules. This proves the claim.
		\end{proof}
	
		\begin{remark}\label{HHWpsf}
			This result shows $H\Z/2\wedge H_W\Z$ is a free left module over $H\Z/2$. It is also free as a right module over $H\Z/2$. Indeed, this statement is equivalent to the map $$H\Z/2_{**}H_W\Z\ra H\Z/2_{**}H_W\Z, \tau \mapsto \tau+\rho\tau_0,$$ being a $k^M_*(k)$-isomorphism. To see that it is one checks what happens to basis elements, that is monomials
			$$\tau_0^{E_0}\tau_1^{E_1}\cdots(\overline{\xi_1}\tau)^\epsilon\overline{\xi_1}^{2R_1}\overline{\xi_2}^{R_2}\overline{\xi_3}^{R_3}\cdots,$$ with $\epsilon,E_i\in \{0,1\}$ and $R_i\in \N_0$. We have $(\tau+\rho\tau_0)^k=\tau^k+T$ with $T$ divisible by powers of $\rho\tau_0$. When multiplying this with a basis monomial, the equations for $\tau_i^2$ tell us that after we expand these into basis monomials, all the summands will have increased index in $E$ or $R$. Therefore they will be greater with respect to a lexicographic order $\tau_i<\overline{\xi_i}$. So the matrix of this map has $1$s on the diagonal and is lower triangular, proving the claim. This is also true for the basis on $H\Z/2_{**}H\Z/2$ by the same proof, see \cite[Cor 6.4]{bachmann-etaperiodic-fields}). 
		\end{remark}
	
		\begin{remark}\label{elements_in_H_H_W}
			We would like to emphasize that $H\Z/2_{**}H_W\Z$ contains all of $H\Z/2_{**}H\Z/2\tau$, hence the relations of Proposition \ref{H_H_W} make sense. It also contains all of $\xi_r^2$. More explicitly, by the conjugation formulas (\ref{hopf_algebroid_structure})
			$$\xi_1^2=\overline{\xi_1}^2, \hspace{1cm} \xi_r^2=\overline{\xi_r}^2+\sum_{i=1}^{r-1}\xi_{r-i}^{2^{i+1}}\overline{\xi_i}^2, \hspace{1cm}\xi_r\tau=\overline{\xi_r}\tau+\sum_{i=1}^{r-1}\xi_{r-i}^{2^i}\overline{\xi_i}\tau.$$
			The first element is in $H\Z/2_{**}H_W\Z$ as a generator. By induction, the second formula shows that all $\xi_i^2$ are in $H\Z/2_{**}H_W\Z$. Similarly, the element $\xi_1\tau$ is in $H\Z/2_{**}H_W\Z$ as a generator, hence the last formula shows by induction that all $\xi_i\tau$ are in $H\Z/2_{**}H_W\Z$. Ideally we would write the relations of Proposition \ref{H_H_W} using an explicit formula in the $\overline{\xi_i}$, $\overline{\xi_1}^2$ $\overline{\xi_1}\tau$, but this seems unfeasible. We must therefore track carefully what is in $H\Z/2_{**}H_W\Z$ and not get confused by multiplication with $\tau$ on the left and right. We mention as an example $\xi_2\tau=\overline{\xi_2}\tau+\overline{\xi_1}^2(\overline{\xi_1}\tau)$.
		\end{remark}
		We will prefer to suppress the generator $\overline{\xi_1}\tau$ as $\tau_0^2-\rho\tau_1$ and instead use the presentation
		\begin{corollary}
			We have an isomorphism of left $H\Z/2_{**}$-algebras 
			$$ H\Z/2_{**} H_W\Z \overset{r^R_*}{\underset{\cong}\lra}H\Z/2_{**}[\tau_0,\ldots,\overline{\xi_1}^2,\overline{\xi_2},\overline{\xi_3}\ldots]/(\ldots) \hra  H\Z/2_{**}H\Z/2,$$
			where the ideal $(\ldots)$ is generated by the relations 
			$$\tau_0^4=\rho^3\tau_2+\rho^2(\overline{\xi_1}^2(\tau_0^2-\rho\tau_1)+\overline{\xi_2}\tau)+\overline{\xi_1}^2\tau^2,\hspace{1cm}\tau_{i}^2=\rho\tau_{i+1}+\xi_{i+1}\tau,\hspace{0.5cm}\text{for }i\geq 1.$$
		\end{corollary}
		
	\subsection{Computation of \texorpdfstring{$\underline{k}^M\wedge H_W\Z$}{kM Homology of Hw}}
		Next on the list is $\underline{k}^M\wedge H_W\Z$. 
		
		By Proposition \ref{*alphaandalpha*formula}, the map $\tau_*^L:H\Z/2_{**}H\Z/2\ra H\Z/2_{**}H\Z/2$ is given by multiplication by $\tau$ on the left. Since $r^R_*:H\Z/2_{**}H_W\Z\hra H\Z/2_{**}H\Z/2$ and $\tau$ is a non-zero divisor in $H\Z/2_{**}H\Z/2$, $\tau$ is also a non-zero divisor in $H\Z/2_{**}H_W\Z$. So by the long exact sequence of the second column of (\ref{grid_H_HW}) we have the short exact sequence
		
		\begin{equation}\label{kMHW}
			0\ra \Sigma^{0,-1}H\Z/2_{**}H_W\Z\xra{{\tau_*^L}}H\Z/2_{**}H_W\Z\xra{{\pi_*^L}} \underline{k}^M_{**}H_W\Z\ra 0.
		\end{equation}
		 It follows immediately
		\begin{proposition}\label{kM_H_W}
			We have as a left $\underline{k}^M_{**}$-algebra
			$$\underline{k}^M_{**}H_W\Z\cong \underline{k}^M_{**}[\tau_0,\ldots,\overline{\xi_1}^2,\overline{\xi_2},\overline{\xi_3}\ldots]/(\ldots),$$
			where the ideal $(\ldots)$ is generated by the relations 
			$$\tau_0^4=\rho^3(\tau_0\overline{\xi_2}+\tau_1\overline{\xi_1}^2+\tau_2),\hspace{1cm}\tau_{r}^2=\rho\tau_{r+1}+\xi_r^2(\tau_0^2-\rho\tau_1)+\rho\tau_0\sum_{i=2}^{r+1}\xi_{r+1-i}^{2^i}\overline{\xi_i},\hspace{0.5cm}\text{for }r\geq 1.$$
		\end{proposition}
		\begin{remark}\label{kMHW_is_psf}
			The description shows that $\underline{k}^M\wedge H_W\Z$ is a free left module over $\underline{k}^M$ with basis given by monomials 
			$$\tau_0^{E_0}\tau_1^{E_1}\cdots\overline{\xi_1}^{2R_1}\overline{\xi_2}^{R_2}\overline{\xi_3}^{R_3}\cdots,$$ with $E_0\in \{0,1,2,3\},E_i\in \{0,1\}$ and $R_i\in \N_0$. We will later use alternative basis monomials $\tau(E)\overline{\xi}(I+2S)$ with $E=(E_0,E_1,\ldots)$, $I=(I_2,I_3,\ldots)$ and $S=(S_1,S_2,\ldots)$ finite sequences with $E_0\in\{0,1,2,3\}$, $E_i\in \{0,1\}$ for $i\geq 1$, $I_j\in \{0,1\}$ and $S_r\in \N_0$.
		\end{remark}
		
	\subsection{Computation of \texorpdfstring{$\underline{K}^W\wedge H_W\Z$}{kW homology of Hw}}
		
		We now want to upgrade the left side of $H\Z/2\wedge H_W\Z$ from $H\Z/2$ to $\underline{K}^W$ through $\underline{k}^M$. We proceed similarly to our computation of $H\Z/2\wedge \underline{K}^W$, but now we apply homology on the other side.
		
		Taking $(-)_{**}H_W\Z$ of (\ref{cofiber_kW_k}) we get a diagram with long exact row
		\begin{equation}\label{*d_longsequence}
			\begin{tikzcd}
				\cdots \arrow[r,"\eta"] & \underline{K}^W_{**}H_W\Z \arrow[r,"{\underline{r}_*^L}"] & \underline{k}^M_{**}H_W\Z \arrow[r, "{\underline{\partial}_*^L}"] \arrow[dr,"{d_*^L}"'] & \Sigma^{2,1} \underline{K}^W_{**}H_W\Z \arrow[r, "\eta"]  \arrow[d, "{\underline{r}_*^L}"] & \cdots \\
				& & & \Sigma^{2,1} \underline{k}^M_{**}H_W\Z.
			\end{tikzcd}
		\end{equation}
		The image of ${\underline{\partial}_*^L}$ is $\eta$-torsion. While we cannot directly access it, we will show with the help of the ${d_*^L}$ that there is no higher $\eta$-torsion. We note that ${d_*^L}\circ {d_*^L}=\underline{r}_*^L\circ\underline{\partial}_*^L\circ \underline{r}_*^L\circ\underline{\partial}_*^L=\underline{r}_*^L\circ 0\circ\underline{\partial}_*^L=0$. One may then guess that the homology of ${d_*^L}$ is a shadow of the $\underline{K}^W_*$-free part of $\underline{K}^W\wedge H_W\Z$. Luckily, it turns out that this is precisely the case. Similarly to $d^R_*$, ${d_*^L}$ is dictated by its lift ${(Sq^2)_*^L}$ to $H\Z/2_{**}H_W\Z\hra H\Z/2_{**}H\Z/2$ since ${\pi_*^L}$ is surjective by (\ref{kMHW}). 
		\begin{notation}\label{index_notation}
			Similar to Section \ref{sec_steenrod}, we let $I=(I_2,I_3,\ldots)$ be a sequence with $I_i\in \{0,1\}$ and almost all $I_i=0$. We denote the zero sequence by $\emptyset$ and denote by $e_i$ the sequence with a single $1$ at the $i$-th position. If not otherwise specified, $I$ and $J$ will in what follows always stand for a sequence starting at $I_2$ and $J_2$ respectively, i.e. $I=(I_2,I_3,\ldots)$ and $J=(J_2,J_3,\ldots)$. We will call such an index set \emph{a binary index set starting at index 2}. We will often think of these index sets as sets of non-zero indices, i.e. we identify $(I_2,I_3,\ldots)=(0,1,0,1,0,0,\ldots)$ with the set $\{3,5\}$. In this sense we may use set operations $\amalg,\cup, \cap, \Delta$, etc., and also the addition and subtraction operations of sequences. 
		\end{notation}
		We compute in $H\Z/2_{**}H\Z/2$ using (\ref{hopf_algebroid_structure}), Proposition \ref{*alphaandalpha*formula} and Proposition \ref{Sq2_on_products}, and define
		\begin{gather*}
			{(Sq^2)_*^L}(\tau_1)=\tau_0,\hspace{1cm} {(Sq^2)_*^L}(\tau_i)=0, \hspace{0.5cm}\text{for }i\neq 1,\\
			{(Sq^2)_*^L}(\overline{\xi}(I))=\sum_{i\in I}\overline{\xi_{i-1}}^2\overline{\xi}(I\setminus i)=:c(I), \hspace{0.7cm} {(Sq^2)_*^L}(\tau_1 \overline{\xi(I)})=\tau_0\overline{\xi}(I)+\tau_1c(I):=c_1(I).		\end{gather*}
		Note $c(e_j)=\overline{\xi_j}^2$ and $c(e_1)=1$ and $c_1(\emptyset)=\tau_0$. Let us recall that the degrees are
		\begin{gather*}
			|\tau_i|=(2^{i+1} -1,2^i -1), \hspace{1cm} |\overline{\xi_i}|=(2^{i+1}-2,2^{i}-1),\\
			|c(I)|=\sum_{i\in I} (2^{j+2}-4,2^{j+1}-2), \hspace{1cm} |c_1(I)|=\sum_{i\in I} (2^{j+2}-1,2^{j+1}-1).
		\end{gather*}
		\begin{lemma}\label{c_relations}
			Let $I,J$ be binary index sets starting at index 2. In $H\Z/2_{**}H_W\Z\subset H\Z/2_{**}H\Z/2$ we have the relations
			\begin{align*}
				c_1(\emptyset)^4&=\tau_0^4=\rho^3(\tau_2+c_1(e_2))+\rho^2\tau\overline{\xi_2}+\tau^2c(e_2),\\
				c(I)c(J)&=\sum_{i\in I\cap J}c(e_i)c(I\Delta J\amalg i)\prod_{j\in I\cap J\setminus i}c(e_{j+1}) +\sum_{i\in I\setminus J}c(e_i)c(I\Delta J\setminus i)\prod_{j\in I\cap J}c(e_{j+1}),\\
				c(I)c_1(J)&=\sum_{i\in I\cap J}c(e_i)c_1(I\Delta J\amalg i)\prod_{j\in I\cap J\setminus i}c(e_{j+1}) +\sum_{i\in I\setminus J}c(e_i)c_1(I\Delta J\setminus i)\prod_{j\in I\cap J}c(e_{j+1}),\\
				c_1(I)c_1(J)&=\tau_0c_1(I\Delta J)\prod_{r\in I\cap J}c(e_{r+1})+(\rho\tau_2+\xi_2\tau)c(I)c(J).
			\end{align*}
		\end{lemma}
		\begin{proof}
			The formula for $c_1(\emptyset)^4$ is simply a restatement of the relation in Proposition \ref{H_H_W}. We prove the second formula. The others follow by similar arguments.
			\begin{align*}
				c(I)c(J)&=\sum_{i\in I,j\in J}\overline{\xi_{i-1}}^2\overline{\xi_{j-1}}^2\overline{\xi}(I\setminus i)\overline{\xi}(J\setminus j)=\sum_{i\in I\cap J}\left[\sum_{j\in J\setminus I\amalg i}+\sum_{j\in I\cap J\setminus i}\right]+\sum_{j\in I\cap J}\sum_{i\in I\setminus J}+\sum_{j\in J\setminus I}\sum_{i\in I\setminus J}\\
				&=\sum_{i\in I\cap J}\sum_{j\in I\Delta J \amalg i}+\sum_{i\in I\setminus J}\sum_{j\in J\setminus I}+\sum_{i\in I\setminus J}\sum_{j\in I\setminus J\setminus i}=\sum_{i\in I\cap J}\sum_{j\in I\Delta J \amalg i}+\sum_{i\in I\setminus J}\sum_{j\in I\Delta J\setminus i}\\
				&=\sum_{i\in I \cap J}\overline{\xi_{i-1}}^2\sum_{j\in I\Delta J \amalg i}\overline{\xi_{j-1}}^2\overline{\xi}(I\Delta J\amalg i\setminus j)\overline{\xi}(I\cap J\setminus i)^2+\sum_{i\in I \setminus J}\overline{\xi_{i-1}}^2\sum_{j\in I\Delta J \setminus i}\overline{\xi_{j-1}}^2\overline{\xi}(I\Delta J\setminus i)\overline{\xi}(I\cap J)^2.
			\end{align*}
			In the third equation we use symmetry to remove the double sum $0=\sum_{i\in I\cap J}\sum_{j\in I\cap J \setminus i}$ as well as to add $0=\sum_{i\in I\setminus J}\sum_{j\in I\setminus J\setminus i}$.
		\end{proof}
		\begin{remark}
			We will often use the notation $c_1(\emptyset)$ and $\tau_0$ interchangeably. In the above we also used $\xi_2\tau$. It can be expressed as $\xi_2\tau=\tau\overline{\xi_2}+\tau_0^2c(e_2)+\rho c_1(e_2)$.
		\end{remark}
		
		The map ${d_*^L}$ is a derivation on $\underline{k}^M_{**}H_W\Z$, thanks to Proposition \ref{Sq2_on_products}. Therefore, its kernel is a subring of $\underline{k}^M_{**}H_W\Z$ and its image is an ideal inside it. The elements we have just computed are not all elements ${d_*^L}$ sends to zero. For example, let $r\in k^M_*(k)$ be $\rho^3$-torsion, then $r{d_*^L}(\tau_0^3\tau_1)=r\tau_0^4=r\rho^3(c_1(e_1)+\tau_2)=0$. Define $b:=\tau_0^3\tau_1$ if there exists $\rho$-torsion in $k^M_*(k)$ and let $b=0$ if $k^M_*(k)$ is $\rho$-torsion free. The degree of $b$ is $|b|=(6,1)$.
		\begin{lemma}\label{image_kernel_of_d}
			Let the base field $k$ satisfy $\rho^3=0$. Then \begin{align*}
				\ker({d_*^L})&=\underline{k}^M_{**}[b,\tau_2,\tau_3,\ldots,c(I),c_1(I)]/(\ldots)\hra \underline{k}^M_{**}H_W\Z,\\
				\im({d_*^L})&=(c(I),c_1(I))\hra \ker({d_*^L}),\\
				H({d_*^L})&=\underline{k}^M_{**}[\tau_2,b,\tau_3,\tau_4,\ldots]/(b^2,\tau_j^2-\rho\tau_{j+1}).
			\end{align*}
			where $I$ is a binary index set starting at index $2$. The ideal $(\ldots)$ is generated by the relations in Lemma \ref{c_relations} mod $\tau$ as well as
			\begin{gather*}
				c(\emptyset)=1,\hspace{1cm} c_1(\emptyset)=\tau_0,\hspace{1cm} \tau_0^4=0,\hspace{1cm}
				\tau_{j}^2=\rho\tau_{j+1}+\xi_{j+1}\tau,\hspace{0.5cm}\text{for }j\geq 2,\\
				b^2=b\tau_0=0,\hspace{1cm}bc(I)=\tau_0^3 c_1(I),\hspace{1cm} bc_1(I)=\rho\tau_0^3(c_1(e_2)+\tau_2)c(I),
			\end{gather*}
		\end{lemma}
		\begin{proof}
			We begin by finding generators of the image. We have seen above that $\tau_i$ lies in $\ker({d_*^L})$ when $i\neq 1$. Since ${d_*^L}$ is a derivation, all squares lie in its kernel as well. Let $I$ be a binary index set starting at index 2 and $S=(S_1,S_2,\ldots)$ a finite index set with $S_i\in \N_0$. It follows that for a basis monomial we have
			$${d_*^L}(\tau(E)\overline{\xi}(I+2S))=\tau(E\setminus 1)\overline{\xi}^2(S){d_*^L}(\tau_1^{E_1}\overline{\xi}(I)).$$
			It thus suffices to show that we can write ${d_*^L}(\tau_1^{E_1}\overline{\xi}(I))$ in the claimed form. This is clear since it is either $c(I)$ or $c_1(I)$. 
			
			We now establish the relations. The relation for $\tau_0^4$ follows because $\rho^3=0$. This immediately implies $b^2=b\tau_0=0$, $bc(I)=\tau_0^3c_1(I)$. Furthermore,
			$$bc_1(I)=\tau_0^3(\rho\tau_2+\xi_2\tau)c(I)=\tau_0^3(\rho\tau_2+\rho c_1(e_2)+\tau_0^2c(e_2))c(I)=\rho\tau_0^3(\tau_2+c_1(e_2))c(I).$$
			
			To see that the kernel is given as the claimed algebra we first note that all the generators are in the kernel of ${d_*^L}$. Now we consider linear combinations of basis elements $\tau(E\setminus 1)\overline{\xi}(I+2S)$. We will show that any such linear combination in the kernel must already be a polynomial in $c(I),\tau(E\setminus 1)$. Since the submodule $\langle \tau(E\setminus 1)\rangle$ is free and ${d_*^L}=id\otimes {d_*^L}$ on $V:=\langle \tau(E\setminus 1)\overline{\xi}(I+2S\neq \emptyset)\rangle\cong\langle \tau(E\setminus 1)\rangle\otimes\langle \overline{\xi}(I+2S\neq \emptyset)\rangle$, we can reduce to linear combinations of $\overline{\xi}(I+2S)$ with $I+2S\neq \emptyset$. Assume $$0={d_*^L}(\sum_{I,S\neq \emptyset}a_{I,S\neq \emptyset}\overline{\xi}(I+2S))=\sum_{I,S\neq \emptyset}a_{I,S}\sum_{i\in I}\overline{\xi_{i-1}}^2\overline{\xi}(I\setminus i)\overline{\xi}^2(S).$$
			Since $\overline{\xi_{i-1}}^2\overline{\xi}(I\setminus i)\overline{\xi}^2(S)$ are linearly independent for varying $i\in I$ and fixed $(I,S)$, for each $i\in I$ there must be another pair $(J_i,S_i)\neq (I,S)$ with index $j_i\in J_i$ canceling $a_{I,S}\overline{\xi_{i-1}}^2\overline{\xi}(I\setminus i)\overline{\xi}^2(S)$. In that case $a_{J_i,S_i}=a_{I,S}$ and $\overline{\xi_{j_i-1}}^2\overline{\xi}(J_i\setminus j_i)\overline{\xi}^2(S_i)=\overline{\xi_{i-1}}^2\overline{\xi}(I\setminus i)\overline{\xi}^2(S)$. Terms $\overline{\xi_{j_i-1}}^2\overline{\xi}(J_i\setminus j_i)\overline{\xi}^2(S_i)$ can only come from $c(I\setminus i \amalg (s+1))\overline{\xi}^2(S- e_s +e_{i-1})$ for some $s\in S$ such that $s+1\notin I$.  Fix $I',S'$ with $a_{I',S'}\neq 0$. Let $I'=\coprod_m D_m$ be a finite partition of $I'$ and let $s_m\in S'$ be distinct elements with $s_m+1\notin I'$. Each $D_m$ is responsible for canceling terms via $c(I'\setminus i \amalg (s_m+1))\overline{\xi}^2(S'- e_{s_m} +e_{i-1})$. We have 
			\begin{align}\label{temp4} 
				\begin{split}
					\sum_m\sum_{i\in D_m}& c(I'\setminus i \amalg (s_m+1))\overline{\xi}^2(S'- e_{s_m} +e_{i-1})\\
					&=\sum_m\sum_{i\in D_m}\sum_{j\in (I'\setminus i)\amalg (s_m+1)}\overline{\xi}_{j-1}^2\overline{\xi}((I'\setminus i\amalg (s_m+1))\setminus j)\overline{\xi}^2(S'-e_{s_m}+e_{i-1})\\
					&= c(I')\overline{\xi}^2(S')+	\sum_m\sum_{i\in D_m}\sum_{j\in I'\setminus i}\overline{\xi}_{j-1}^2\overline{\xi}((I'\setminus i\amalg (s_m+1))\setminus j)\overline{\xi}^2(S'-e_{s_m}+e_{i-1}).
				\end{split}
			\end{align}
			The terms $\overline{\xi}_{j-1}^2\overline{\xi}((I'\setminus i\amalg (s_m+1))\setminus j)\overline{\xi}^2(S'-e_{s_m}+e_{i-1})$ are only linearly dependent for pairs $(i,j)$, $(j,i)$ since the term with singular powers is different otherwise. It follows that the $\sum_m\sum_{i\in D_m}\sum_{j\in I'\setminus i}$ term in the third line of (\ref{temp4}) only vanishes if $D_m=I$ for some $m$. We conclude 
			$$\sum_{I,S\neq \emptyset}a_{I,S\neq \emptyset}\overline{\xi}(I+2S)=a_{I',S'}c(I'\amalg (s_m+1))\overline{\xi}(S'-e_{s_m})+R,$$
			where $R$ now contains less summands with $a_{I,S}\neq 0$. Inductively we remove all terms until the claim follows, i.e. $\ker({d_*^L}_{|_V})=\im({d_*^L}_{|_V})$. 
			
			We now turn to a general element in the kernel of ${d_*^L}$. In the following all sums go over $E, I, S$ as in Remark \ref{kMHW_is_psf} and we suppress these index sets in the coefficients of the basis elements
			\begin{align*}
				\sum_{E,I,S}a_{E,I,S}{d_*^L}(\tau(E)\overline{\xi}(I+2S))&=\sum_{E,I,S}a_{E,I,S}{d_*^L}(\tau_1^{E_1}\overline{\xi}(I))\tau(E\setminus 1)\overline{\xi}(2S)\\
				&=\sum_{E_0=0}\left[a_{E_0=0,E_1=0}c(I)+a_{E_0=0,E_1=1}c_1(I)\right]\tau(E\setminus \{0,1\})\overline{\xi}(2S)\\
				& +\sum_{E_0=1}\left[a_{E_0=1,E_1=0}\tau_0c(I)+a_{E_0=1,E_1=1}\tau_0c_1(I)\right]\tau(E\setminus \{0,1\})\overline{\xi}(2S)\\
				& +\sum_{E_0=2}\left[a_{E_0=2,E_1=0}\tau_0^2c(I)+a_{E_0=2,E_1=1}\tau_0^2c_1(I)\right]\tau(E\setminus \{0,1\})\overline{\xi}(2S)\\
				& +\sum_{E_0=3}\left[a_{E_0=3,E_1=0}\tau_0^3c(I)+a_{E_0=3,E_1=1}\tau_0^3\tau_1c(I)\right]\tau(E\setminus \{0,1\})\overline{\xi}(2S).
			\end{align*}
			We assume that the above sum is zero. In the last row we used the relation $\tau_0^3c_1(I)=b c(I)$. Note that the very first sum $\sum_{E_0=0}a_{E_0=0,E_1=0}c(I)$ contains the only elements not divisible by $\tau_0$ or $\tau_1$. It is thus independent of all the others. We have already shown that any sum of that form comes from an element in the image of ${d_*^L}$. We may thus assume that $a_{E_0=0,E_1=0}=0$. Similarly, the very last sum contains the only elements divisible by $\tau_0^3\tau_1$. By the same logic we may assume $a_{E_0=3,E_1=1}=0$. The elements $\tau_0^2c_1(I)=\tau_0^3\overline{\xi}(I)+\tau_0^2\tau_1c(I)$ are the only ones that contain summands divisible by $\tau_0^2\tau_1$. With the same argument we deduce $a_{E_0=2,E_1=1}=0$ which then leads to $a_{E_0=3,E_1=0}=0$. We repeat the argument with elements divisible by $\tau_0\tau_1$ to deduce $a_{E_0=1,E_1=1}=0$. This leads to $a_{E_0=2,E_1=0}=0$. The last two vanish by the same reason. This shows that the kernel is generated by elements as claimed. The relations we have established reduce every possible relation to terms of low powers in the generators. One can check that the remaining monomials are linearly independent in $k^M_{**}H_W\Z$ and thus generate the ideal $(\ldots)$.
			
			Taking quotients we see immediately that $H({d_*^L})$ is generated as claimed. We have $\xi_1\xi_i\tau\in H\Z/2_{**}H_W\Z$ and thus also in $\underline{k}^M_{**}H_W\Z$. Therefore, $\xi_{j}\tau={d_*^L}(\xi_1\xi_j\tau)$, for $j\geq 2$, is in the image of ${d_*^L}$. This shows the relation $\tau_j^2-\rho\tau_{j+1}=0$ in $H({d_*^L})$. The relation $b^2=0$ already holds before taking quotients. The last relation $\rho^3\tau_2$ follows from the relation of $\tau_0^4$. Since all monomials with no higher powers in $b$ and $\tau_{j}$, for $j\geq 2$, are linearly independent in $\underline{k}^M_{**}H_W\Z$, these are all relations in $H({d_*^L})$. 
		\end{proof}
		\begin{remark}
			The condition $\rho^3=0$ holds for any field of odd characteristic. It further holds for fields that are extensions of totally imaginary number fields of transcendence degree at most 1, by \cite[II.4.4, Prop. 11 and 13]{serre_galoiscohomology} and the resolution of the Milnor conjecture. We note that all fields with $\rho^3=0$ are nonreal \cite[Cor. 6.20]{lamintrotoquadforms}.
		\end{remark}

		In the next two Propositions our base field will be assumed to satisfy $\rho^3=0$ and $\vcd_2(k)<\infty$. In particular $k$ is nonreal. By Corollary \ref{nonrealfieldsdontneedcompletion} and Proposition \ref{ourpsectraareveff}, all combinations of smash products of the spectra $H_W,H\Z/2$, $\underline{K}^W$, $\underline{k}^M$ are $\eta$-complete under this assumption.
	
		\begin{proposition}\label{no_higher_eta-torsion}
			Let the base field $k$ satisfy $\vcd_2(k)<\infty$, $\rho^3=0$ and $k^M_4(k)=0$. Then we have $\im({\underline{r}_*^L})=\ker({d_*^L})$ and $\underline{K}^W_{**} H_W\Z$ has no higher $\eta$-torsion.
		\end{proposition}
		\begin{proof}
			Recall (\ref{*d_longsequence}) 
			\[\begin{tikzcd}
			\cdots \arrow[r,"\eta"] & \underline{K}^W_{**} H_W\Z \arrow[r,"{\underline{r}_*^L}"] & \underline{k}^M_{**}H_W\Z \arrow[r, "{\underline{\partial}_*^L}"] \arrow[dr,"{d_*^L}"'] & \Sigma^{2,1} \underline{K}^W_{**} H_W\Z \arrow[r, "\eta"]  \arrow[d, "{\underline{r}_*^L}"] & \cdots \\
			& & & \Sigma^{2,1} \underline{k}^M_{**}H_W\Z.
			\end{tikzcd}\] 
			Let $x\in \underline{K}^W_{**} H_W\Z$ be higher $\eta$-torsion, i.e. $x={\underline{\partial}_*^L}(y)=a\eta^k$ with $k\geq 1$. Since $x$ is divisible by $\eta$ iff it lies in the kernel of ${\underline{r}_*^L}$, we know that $y\in \ker({d_*^L})$. Adding elements in $\im({d_*^L})\subset \ker({\underline{\partial}_*^L})$ to $y$ does not change their image under ${\underline{\partial}_*^L}$. Therefore, without loss of generality, $y$ is a polynomial in $b$ and $\tau_j$ for $j\geq 2$ by Lemma \ref{image_kernel_of_d}. Since ${\underline{r}_*^L}$ is a ring map and we want to show that $x=0$, i.e. $y$ is in the image of $\underline{r}_*$, it suffices to show that $b,\tau_j\in \im({\underline{r}_*^L})$ for $j\geq 2$. Indeed, if they all are in the image, then so is any polynomial in them. 
			
			Let us first assume $y=\tau_j$ with $j\geq 2$. Since $\underline{K}^W$ is a connective $E_\infty$-ring spectrum and $\underline{K}^W\wedge H_W$ is $\eta$-complete, each $i$-line $\pi_i(\underline{K}^W\wedge H_W)_*$ is $\eta$-complete as a $\underline{K}^W_*$-module by Theorem \ref{completion_and_htpy_groups}. In particular, every $i$-line $\pi_i(\underline{K}^W\wedge H_W)_*$ is separated with respect to $(\eta)$. Thus, we can pick $k\geq 1$ maximal with $x={\underline{\partial}_*^L}(\tau_j)=a\eta^k$, i.e. s.t. $z:={\underline{r}_*^L}(a)\neq 0$. Clearly, $z\in \ker({d_*^L})$. $\tau_j$ lives on the $2^j$-line, $b$ lives on the $2^2+1=5$-line and $z$ on the $2^j-1$-line. By degree reasons, $z$ maps to zero in $H({d_*^L})$, i.e. it lies in the image of ${d_*^L}$. Therefore, there are $u\in\underline{k}^M_{**}H_W\Z $ and $v\in \pi_{**}L$ such that $$z={\underline{r}_*^L}(a)={\underline{r}_*^L}({\underline{\partial}_*^L}(u)) \hspace{3mm}\leadsto \hspace{3mm} a-{\underline{\partial}_*^L}(u)\in \ker({\underline{r}_*^L}) \hspace{3mm} \longleftrightarrow \hspace{3mm}  a-{\underline{\partial}_*^L}(u)=\eta^mv.$$
			Here we again choose $m\geq 1$ maximal. Since the image of ${\underline{\partial}_*^L}$ is $\eta$-torsion, this implies:
			$$x=\eta^k a=\eta^k(\eta^m v+{\underline{\partial}_*^L}(u))=\eta^{k+m}v$$
			contradicting maximality of $k$. We conclude $x=0$ and thus $\tau_j\in \im({\underline{r}_*^L})$.
			
			Let now $y=b=\tau_0^3\tau_1$. In this case $x$ lives on the $4$-line, just like $\tau_2$. The only possibility is thus $z=s \tau_2$ for $s\in k^M_n(k)$. We have $|x|=(4,0)$ and $|\tau_2|=(7,3)$ and thus $n\geq 4$. By our assumptions on the base field, we must have $s=0$. We reduce $x=0$ and thus $b \in \im({\underline{r}_*^L})$ by the non-existence of infinitely $\eta$-divisible elements, as before. 
		\end{proof}
	
		We thus see that lifts of monomials in $\tau_j$, for $j\geq 2$, correspond to an $\eta$-torsion free part and $\eta$-torsion corresponds to lifts of $\im({d_*^L})$. We can lift $\eta$-torsion canonically and $\underline{r}_*^L$ maps it isomorphically to $\im(d_*^L)$. We will thus abuse notation and denote $\eta$-torsion in $\underline{K}^W_{**} H_W\Z$ by its image under $\underline{r}_*^L$. More precisely, we will denote $${\underline{\partial}_*^L}(\tau_1)=\tau_0=c_1(\emptyset), \hspace{1cm} {\underline{\partial}_*^L}(\overline{\xi}(I))=c(I) ,\hspace{1cm} {\underline{\partial}_*^L}(\tau_1\overline{\xi}(I))=c_1(I).$$ 
		Moreover, we choose ${\underline{r}_*^L}$ lifts $t_j$ of $\tau_j$ for $j\geq 2$ and $s$ of $b$. They exist by the previous Proposition. Since ${\underline{r}_*^L}$ is a left $\underline{k}^M_*$-algebra map, it is clear that the product of two lifts is a lift of the product. Using this property we have lifts of all elements in $\im({\underline{r}_*^L})$. Note that our choice of lifts $t_j$ and $s$ is in general not canonical and only up to the ideal generated by $\eta$.
	
		\begin{proposition}\label{K_W2H_Weta}
			Assumptions as in Proposition \ref{no_higher_eta-torsion}, we have an isomorphism of left $\underline{K}^W_{**}$-algebras $$\underline{K}^W_{**} H_W\Z \cong \underline{K}^W_{**} [s,t_2,t_3,\ldots,c(I),c_1(I)]/(\ldots),$$
			where $I$ is a binary index set starting at index $2$. The ideal of $\eta$-torsion is generated by $c(I),c_1(I)$. The ideal $(\ldots)$ is generated by the relations in Lemma \ref{c_relations} mod $\tau$ as well as the following
			\begin{gather*}
				c(\emptyset)=1,\hspace{1cm} c_1(\emptyset)=\tau_0,\hspace{1cm} \tau_0^4=0,\hspace{1cm}t_j^2=\rho t_{j+1}+\xi_{j+1}\tau+O(\eta),\\
				s^2=s\tau_0=0\hspace{1cm}bc(I)=\tau_0^3 c_1(I),\hspace{1cm} sc_1(I)=\rho\tau_0^3(c_1(e_2)+\tau_2)c(I).
			\end{gather*}
		\end{proposition}
		\begin{proof}
			Since $\underline{K}^W$ is a connective $E_\infty$-ring spectrum and $\underline{K}^W\wedge H_W\Z$ is $\eta$-complete, each $i$-line $\pi_i(\underline{K}^W\wedge H_W\Z)_*$ is $\eta$-complete as a $\underline{K}^W_{**}$-module by Theorem \ref{completion_and_htpy_groups} and $\underline{K}^W_{**}$ is $\eta$-complete. By Proposition \ref{no_higher_eta-torsion}, $\underline{r^L_*}(\pi_i(\underline{K}^W\wedge H_W\Z)_*)$ is finitely generated over $\underline{k}^M_{**}$ by certain monomials in $b,\tau_j,c(I),c_1(I)$ as described in Lemma \ref{image_kernel_of_d}. By Nakayama's Lemma \cite[Thm. 8.4]{matsumura_CRT}, their lifts generate $\pi_i(\underline{K}^W\wedge H_W\Z)_*$ as a left $\underline{K}^W_{**}$-module. Hence, $\underline{K}^W_{**} H_W\Z$ is generated as left $\underline{K}^W_{**}$-algebra by the lifts $s,t_j,c(I),c_1(I)$. 
			
			By Lemma \ref{no_higher_eta-torsion} there is no higher $\eta$-torsion, hence $\eta$-torsion is generated as claimed. The map $\underline{r}_*^L$ maps $\eta$-torsion isomorphically to $\im(d_*^L)$. Any relations in Lemma \ref{image_kernel_of_d} involving only $\eta$-torsion thus hold verbatim in $\underline{K}^W_{**} H_W$. Relations that involve elements that are not $\eta$-torsion only hold up to the ideal generated by $\eta$. These relations thus have $\eta$-residue $O(\eta)$. More precisely, $\eta$-residue is present in the relations involving $t_j^2-\rho t_{j+1}$ .
		\end{proof}
		\begin{example}\label{uniqueness_of_lifts}
			Let $k$ be a field of characteristic not $2$ with $\vcd_2(k)<\infty$ and $k_2^M(k)=0$, then $K^W_2(k)=0$ by \cite[Cor. 6.19]{lamintrotoquadforms}. Therefore, in this case $(\eta)_1=0$ in $K^W_*(k)$. It follows that $\rho\tau_j$ have a unique lift over these fields. Examples of such fields are finite fields $\mathbb{F}_q$ with $q$ odd as well as quadratically closed fields.
		\end{example}
		\begin{remark}\label{eta-completed_version}
			By the same arguments, one can prove $\eta$-completed versions of Proposition \ref{no_higher_eta-torsion} and Proposition \ref{K_W2H_Weta} for base fields $k$ without the assumption on finite virtual cohomological $2$-dimension. To do so one additionally uses the fact that $\eta$-completion of connective spectra is given by $H\Z$-localization \cite[5.2]{localizationsinmotivic}, which preserves cofiber sequences and ring structures. Hence, $(\underline{K}^W\wedge H_W)^\wedge_\eta$ is still a ring spectrum and that one still has a diagram (\ref{*d_longsequence}) after $\eta$-completion. In this case one gets an isomorphism of left $(\underline{K}^W_{**})^\wedge_\eta$-algebra
			$$\pi_{**}((\underline{K}^W\wedge H_W)^\wedge_\eta) \cong (\underline{K}^W_{**})^\wedge_\eta [s,t_2,t_3,\ldots,c(I),c_1(I)]/(\ldots),$$
			where $\eta$-torsion and the ideal $(\ldots)$ have the same generators and relations as in Proposition \ref{K_W2H_Weta}.
		\end{remark}
		\begin{remark}
			We hope that a choice of orientation $MSL\ra H_W\Z$ will lead to a consistent choice of lifts and remove $\eta$-residues over general fields, similarly to how it was done for $\kw$ in \cite{bachmann-etaperiodic-fields} or \cite{ananyevskiy_stableopinderwitt}.
		\end{remark}
	
	\section{\texorpdfstring{$H_W\Z\wedge H_W\Z$}{Main Result}}
		We now come to $H_W\Z\wedge H_W\Z$. From the pullback square (\ref{H_Wwedgefundamental_pullbacksquare}) we obtain, by similar arguments as in the proof of Lemma \ref{H**HW_square}, a pullback square
		\begin{equation}\label{final_pullback_square}
			\begin{tikzcd}
			{{H_W\Z}_{**}H_W\Z} \ar[r,"{r_*^L}"] \ar[d,twoheadrightarrow,"{(\pi_W)_*^L}"] & H\Z/2_{**}H_W\Z \ar[d,twoheadrightarrow,"{\pi_*^L}"]\\
			{\underline{K}^W_{**}H_W\Z} \ar[r,"{\underline{r}_*^L}"] & \underline{k}^M_{**}H_W\Z.
			\end{tikzcd}
		\end{equation}
		By this pullback square we will get generators for $\tau_0,s,t_j,c(I),c_1(I)$ as well as $\tau \overline{\xi}(I)$ and $\tau\tau_1\overline{\xi}(I)$ for all binary index sets $I$, starting at index $2$. These lie in the kernel of ${(\pi_W)_*^L}$ and come from $\tau_*^L:H\Z/2_{**}H_W\Z\ra{H_W\Z}_{**}H_W\Z$. We denote $$t(I):={\tau_*^L}(\overline{\xi}(I)),\hspace{1cm}t_1:=t_1(\emptyset):=\tau\tau_1,\hspace{1cm}t_1(I):={\tau_*^L}(\tau_1\overline{\xi}(I)),$$ Note that the image of ${\tau_*^L}$ is $\eta$-torsion, since it comes from elements in $H\Z/2_{**}H_W\Z\hra H\Z/2_{**}H\Z/2$. Remember that we have a diagram 
		\[\begin{tikzcd}
			\Sigma^{0,-1} H\Z/2_{**}H_W\Z \ar[r,"id"] \ar[d,"{\tau_*^L}"] & \Sigma^{0,-1}H\Z/2_{**}H_W\Z \ar[d,"{\tau_*^L}"]\\
			{H_W\Z}_{**}H_W\Z \ar[r,"{r_*^L}"] & H\Z/2_{**}H_W\Z
		\end{tikzcd}\] 
		which we get from the top right square of (\ref{grid_H_HW}). This means that we can determine relations of $\eta$-torsion inside ${H_W\Z}_{**}H_W\Z$ simply in $H\Z/2_{**}H_W\Z$ instead.
		\begin{lemma}\label{t_relations}
			Let $I,J$ be binary index sets starting at index 2. In ${H_W\Z}_{**}H_W\Z$ we have the relations
			\begin{align*}
				t(I)t(J)&=t(I\Delta J)\prod_{i\in I\cap J}c(e_{i+1}),\\
				t(I)t_1(J)&=\tau t_1(I\Delta J)\prod_{i\in I\cap J}c(e_{i+1}),\\
				t_1(I)t_1(J)&=\tau(\rho\tau_2+\xi_2 \tau)t(I\Delta J)\prod_{i\in I\cap J}c(e_{i+1}),\\
				c(I)t(J)&=\sum_{i\in I\cap J}c(e_i)t(I\Delta J\amalg i)\prod_{j\in I\cap J\setminus i}c(e_{j+1})+\sum_{i\in I\setminus J}c(e_i)t(I\Delta J\setminus i)\prod_{j\in I \cap J}c(e_{j+1}),\\
				c_1(I)t(J)&=\tau_0t(I\Delta J)\prod_{i\in I\cap J}c(e_{i+1})+\sum_{i\in I\cap J}c(e_i)t_1(I\Delta J\amalg i)\prod_{j\in I\cap J\setminus i}c(e_{j+1})\\
				&\hspace{6cm}+\sum_{i\in I\setminus J}c(e_i)t_1(I\Delta J\setminus i)\prod_{j\in I \cap J}c(e_{j+1}),\\
				c_1(I)t_1(J)&=\tau_0t_1(I\Delta J)\prod_{i\in I\cap J}c(e_{i+1})+(\rho\tau_2+\xi_2 \tau)\bigg[\sum_{i\in I\cap J}c(e_i)t(I\Delta J\amalg i)\prod_{j\in I\cap J\setminus i}c(e_{j+1})\\
				&\hspace{6cm}+\sum_{i\in I\setminus J}c(e_i)t(I\Delta J\setminus i)\prod_{j\in I \cap J}c(e_{j+1})\bigg].
			\end{align*}
		\end{lemma}
		\begin{proof}
			The formulas follow from straightforward calculations similar to Lemma \ref{c_relations}.
		\end{proof}
	
		\begin{remark}
			In the above Lemma we used $\xi_2\tau$. It can be written as $\xi_2\tau=t(e_2)+\tau_0^2c(e_2)+\rho c_1(e_2)$. One can similarly express all $\xi(I)\tau$ in terms of $t(I),c(I),c_1(I)$ using the formulas in Remark \ref{elements_in_H_H_W}. This can be seen from the fact that in $k^M_{**}H_W\Z$ $\xi(I)\tau$ lies in the image of ${d_*^L}$ via $\xi(I)\tau={d_*^L}(\xi(I)(\xi_1\tau))$, since $H\Z/2_{**}H\Z/2\tau\subset H\Z/2_{**}H_W\Z$. Hence, by Lemma \ref{image_kernel_of_d} it can be written as claimed.
		\end{remark}
		
		\begin{theorem}\label{mainthm}
			Let the base field $k$ be an extension of a field $F$ satisfying $k^M_2(F)=0$. Then we have an isomorphism of right ${H_W\Z}_{**}$-algebras
			$${H_W\Z}_{**} {H_W\Z} \cong {H_W\Z}_{**}[s,t_2,t_3,\ldots,c(I),t(I),c_1(I),t_1(I)]/(\ldots),$$
			where $I$ is a binary index set starting at index $2$.  The ideal of $\eta$-torsion is generated by $c(I),c_1(I),t(I),t_1(I)$. The ideal $(\ldots)$ is generated by the relations in Lemma \ref{c_relations}, Lemma \ref{t_relations} and
			\begin{gather*}
			c(\emptyset)=1,\hspace{1cm} \tau_0^4=c(e_{2})\tau^2,\hspace{1cm}t_j^2=\rho t_{j+1}+\xi_{j+1}\tau,\hspace{1cm}s^2=\tau_0^2(\rho t_2+\xi_2\tau)c(e_2)\tau^2,\\
			sc_1(I)=\tau_0^3(\rho t_2+\xi_2\tau)c(I)+\tau t_1(I)c(e_2),\hspace{1cm}t(I)s=t_1(I)\tau_0^3,\hspace{1cm} t_1(I)s=\tau_0^3t(I)(\rho t_2+\xi_2\tau).
			\end{gather*}
			The degrees here are
			\begin{gather*}
				|t_1|=(3,0), \hspace{1cm} |s|=(5,0),\hspace{1cm}|t_i|=(2^{i+1}-1,2^i-1),\\
				|c(I)|=\sum_{i\in I} (2^{j+2}-4,2^{j+1}-2), \hspace{1cm} |c_1(I)|=\sum_{i\in I} (2^{j+2}-1,2^{j+1}-1),\\
				|t(I)|=(0,-1)+\sum_{i\in I} (2^{i+1}-2,2^{i}-1), \hspace{1cm} |t_1(I)|=(3,0)+\sum_{i\in I} (2^{i+1}-2,2^{i}-1).
			\end{gather*}
		\end{theorem}
		\begin{proof}
			By Remark \ref{cellularity} we may check this on homotopy groups $\pi_{p,q}$. By Lemma \ref{pullback_preservesnicespectra} we can reduce to $F$ being our base field. If the characteristic of $F$ is positive, then $F$ is an extension of $\mathbb{F}_q$ with $q$ odd and we may further reduce to $\mathbb{F}_q$. Assume the characteristic of $F$ is $0$. Since $\rho^2=0$ in $k^M_2(F)$, we have $-1=a^2+b^2$ in $F$ \cite[Cor. 6.20]{lamintrotoquadforms}. Hence, $F$ is a field extension of $\Q(a,b)$. Since this is a field extension of transcendence degree at most $1$, we have $cd_2(\Q(a,b))\leq 3$ \cite[II.4.4, Prop. 11 and 13]{serre_galoiscohomology}. Thus, by the resolution of the Milnor conjecture, $\Q(a,b)$ satisfies the assumptions of Proposition \ref{K_W2H_Weta} and we may reduce to that situation.
			
			At the beginning of this section we saw, with the help of the pullback square (\ref{final_pullback_square}), that the given generators generate. The relations that are left are straightforward and can be imported from $H\Z/2_{**}H_W\Z$ and $\underline{K}^W\wedge H_W\Z$, keeping in mind $\rho^2=0$. By Example \ref{uniqueness_of_lifts} we have unique lifts for $\rho\tau_j$ and thus $\tau_j^2$. Together with the fact that we have no higher $\eta$-torsion, we conclude that there are no $\eta$-residue terms in the relations. Linear independence of remaining monomials can now simply be checked in $H\Z/2_{**}H_W\Z$.
		\end{proof}
		\begin{remark}
			We can conjugate all the elements to get an isomorphism of left ${H_W\Z}_{**}$-algebras.
		\end{remark}
		\begin{remark}
			The assumptions in Theorem \ref{mainthm} include fields $k$ of odd characteristic as well as fields that are extensions of fields $F$ with $\cd_2(F)=1$. For example, this includes all extensions of quadratically closed fields, in particular the complex numbers $\C$, and all extensions of $\mathbb{F}_p$ with $p\neq 2$.
		\end{remark}
		\begin{remark}
			Let $F$ be a field of characteristic not $2$ with assumptions as in Proposition \ref{no_higher_eta-torsion}, i.e. $\vcd_2(F)<\infty$, $\rho^3=0$ and $k^M_4(k)=0$. Let the base field $k$ be a field extension of $F$. Then the arguments of the proof of Theorem \ref{mainthm} prove a version with the only difference being $\eta$-residues in the relations for $t_j^2$. If one further drops the assumption $\vcd_2(F)<\infty$, then one can prove an $\eta$-completed version as in Remark \ref{eta-completed_version}.
		\end{remark}
		Let $HW=H_W\Z[\eta^{-1}]$. When we invert $\eta$, the coefficient ring becomes 
		$$H_W\Z_{**}[\eta^{-1}]\cong \underline{K}^W_{**}[\eta^{-1}]\cong W(k)[\eta^{\pm 1}]$$
		by \cite[Cor. 3.11, Remark 3.12]{morela1algtopoverfield}. We can normalize the generators $t_j$ and $s$ to live in motivic weight $0$ via $x_j:=\eta^{1-2^j}t_j$ and $y:=\eta^{-1}s$. 
		\begin{corollary}\label{eta_inverted_algebra}
			Let the base field $k$ be an extension of a field $F$ satisfying $k^M_2(F)=0$. Then we have an isomorphism of $W(k)[\eta^{-1}]$-algebras
			$${HW}_{**}HW\cong W(k)[\eta^{\pm 1}][x_2,y,x_3,x_4,\ldots]/(y^2,x_j^2-2x_{j+1}),$$
			where $|x_j|=(2^j,0)$ and $|y|=(5,0)$.
		\end{corollary} 
		Over fields with the given assumptions, this implies the result \cite[Cor. 8.20]{bachmann-etaperiodic-fields} by Bachmann and Hopkins.